\documentclass[reqno,12pt]{amsart}
\def\vint_#1{\mathchoice%
          {\mathop{\kern 0.2em\vrule width 0.6em height 0.69678ex depth -0.58065ex
                  \kern -0.8em \intop}\nolimits_{\kern -0.4em#1}}%
          {\mathop{\kern 0.1em\vrule width 0.5em height 0.69678ex depth -0.60387ex
                  \kern -0.6em \intop}\nolimits_{#1}}%
          {\mathop{\kern 0.1em\vrule width 0.5em height 0.69678ex
              depth -0.60387ex
                  \kern -0.6em \intop}\nolimits_{#1}}%
          {\mathop{\kern 0.1em\vrule width 0.5em height 0.69678ex depth -0.60387ex
                  \kern -0.6em \intop}\nolimits_{#1}}}
\def\vintslides_#1{\mathchoice%
          {\mathop{\kern 0.1em\vrule width 0.5em height 0.697ex depth -0.581ex
                  \kern -0.6em \intop}\nolimits_{\kern -0.4em#1}}%
          {\mathop{\kern 0.1em\vrule width 0.3em height 0.697ex depth -0.604ex
                  \kern -0.4em \intop}\nolimits_{#1}}%
          {\mathop{\kern 0.1em\vrule width 0.3em height 0.697ex de pth -0.604ex
                  \kern -0.4em \intop}\nolimits_{#1}}%
          {\mathop{\kern 0.1em\vrule width 0.3em height 0.697ex depth -0.604ex
                  \kern -0.4em \intop}\nolimits_{#1}}}

\usepackage{a4wide}
\usepackage{mathtools}
\usepackage{amsthm}
\usepackage{amsmath}
\usepackage{amssymb}
\usepackage{dsfont}
\usepackage{color}
\usepackage[obeyFinal]{todonotes}
\usepackage{csquotes}
\usepackage{graphicx}
\usepackage{epsfig,color}
\usepackage{enumitem}
\usepackage[capitalise]{cleveref}
\usepackage{autonum}
\numberwithin{equation}{section}
\newtheorem{theorem}{Theorem}[section]
\newtheorem{lemma}[theorem]{Lemma}
\newtheorem{corollary}[theorem]{Corollary}
\newtheorem{proposition}[theorem]{Proposition}
\theoremstyle{definition}
\newtheorem{definition}[theorem]{Definition}
\newtheorem{remark}[theorem]{Remark}
\newtheorem{convention}[theorem]{Convention}

\renewcommand{\d}{{\mathrm{\,d}}}
\newcommand{\R}{\mathbb{R}}

\newcommand{\p}{\mathcal{P}}

\newcommand{\spt}{\mathrm{supp}}
\newcommand{\opt}{\mathrm{Opt}}
\newcommand{\optgeo}{\mathrm{OptGeo}}

\newcommand{\loc}{\mathrm{loc}}
\newcommand{\cd}{\mathrm{CD}}
\newcommand{\im}{\mathrm{Im}}
\newcommand{\m}{\mathfrak{m}}
\newcommand{\q}{\mathfrak{q}}
\newcommand{\geo}{\mathrm{Geo}}

\newcommand{\mcp}{\mathrm{MCP}}
\renewcommand{\L}{\mathcal{L}}

\begin{document}

\begin{abstract}
We establish the local-to-global property of the synthetic curvature-dimension condition for essentially non-branching locally finite metric-measure spaces, extending the work [F. Cavalletti, E. Milman \textit{Invent. Math.} 226 (2021), no. 1, 1-137].
\end{abstract}
\title[The Globalization Theorem for $\cd (K, N)$
on locally finite Spaces]{The Globalization Theorem for $\cd (K, N)$
on locally finite Spaces}

\author{Zhenhao Li}

\address{Fakultät für Mathematik \\
         Universität Bielefeld \\
         Postfach 100131 \\
         DE-33501 Bielefeld \\
         Germany}
\email{zhenhao.li@math.uni-bielefeld.de}

\thanks{}
\subjclass[2000]{}

\maketitle
\section{Introduction}\label{sec:Introduction}
The \textit{curvature-dimension condition}, or shortly $\cd (K, N)$ on metric-measure spaces $(X,d,\m)$, was introduced by Lott-Villani and Sturm in the seminal papers \cite{lott-villani, sturm2006--1, Sturm2006}. 

A natural but longstanding question is whether such a synthetically defined condition can be checked locally.
Cavalletti-Milman's recent paper \cite{cavalletti2021globalization} gives a
positive answer to this globalization problem under the assumption $\mathfrak{m}(X)=1$, which was conjectured to be merely technical there.
In this paper, we extend this result to infinite-volume spaces.

\begin{theorem}[Local-to-Global property]\label{thm:local-to-global}
  Let $(X, d, \mathfrak{m})$ be an essentially non-branching metric-measure space\footnote{In the following sections, we will use the abbreviations \textit{e.n.b.} and \textit{m.m.s.} for \textit{essentially non-branching} and \textit{metric-measure space}, resp.} with a locally finite Borel measure $\mathfrak{m}$. 
  Assume that $(\spt(\mathfrak{m}), d)$ is a length space. 
  Then if $(X, d, \mathfrak{m})$ verifies
  $\cd_{\loc} (K, N)$ for $K \in \mathbb{R}$ and $N \in (1, \infty)$, it verifies $\cd (K, N)$.
\end{theorem}
Here immediately follow several useful equivalence results once we apply \cref{thm:local-to-global} to Section 13.1 and 13.2 in \cite{cavalletti2021globalization}.
\begin{corollary}
    Let $(X,d,\m)$ be a metric-measure space with a locally finite Borel measure $\mathfrak{m}$. Then\footnote{We refer to \cite[Section 13]{cavalletti2021globalization} for definitions of the following variants of curvature-dimension conditions.}
    \begin{itemize}
        \item if $(X,d,\m)$ is essentially non-branching, it holds $\cd^{*}(K,N)$ if and only if it holds $\cd(K,N)$;
        \item $(X,d,\m)$ holds $\mathrm{RCD}^{*}(K,N)$ if and only if it holds $\mathrm{RCD}(K,N)$;
        \item if $(\mathrm{supp}(\m),d)$ is a length space, it holds $\mathrm{RCD}_{\loc}(K,N)$ if and only if it holds $\mathrm{RCD}(K,N)$.
    \end{itemize}
\end{corollary}
In \cite{cavalletti2021globalization}, Cavalletti and Milman introduced the $\cd^1(K,N)$ condition on finite-volume spaces, which roughly requires transport rays of signed distance functions to hold the one-dimensional $\cd(K,N)$.
Then they showed that under suitable assumptions, $\cd^1(K,N)$ implies $\cd(K,N)$.
Similarly, in this paper we tailor the definition of $\cd^1(K,N)$, adapting it to the infinite-volume situation by assuming conditional measures to be uniformly-locally finite. 
Then we split the problem into two independent ones: $\cd_{\loc}(K,N)\Rightarrow \cd^1(K,N)$ and $\cd^1(K,N)\Rightarrow \cd(K,N)$.

For the first part, we normalize the reference measure as in \cite{Cavalletti-Mondino2020} and show that the needle/ray-decomposition developed in \cite{Bianchini,cavalletti2013monge} still localizes the curvature-dimension condition to rays. 
For the second part, we show under the given definition, $\cd^1(K,N)$ space is locally finite, geodesic and satisfying $\mcp(K,N)$. 
Then we briefly present the strategy and arguments fulfilling the implication of $\cd(K,N)$ in locally finite spaces, which is basically the same as in \cite{cavalletti2021globalization} under modifications. 
Indeed, the validity is ensured basically by three aspects: (1) owing to the local finiteness of conditional measures and the properness of the space, problems are reduced to the finite-volume case by taking exhaustion by compacts subsets; (2)
$\cd(K,N)$ is reduced to a path-wise inequality along Kantorovich geodesics by the non-branchingness, hence the one-dimensional analysis in \cite[Part III]{cavalletti2021globalization} is not affected by the global infinity of $\m$; (3) temporal derivatives of potentials, investigated in \cite[Part I]{cavalletti2021globalization}, do not rely on the measure structure.

Accordingly, the rest of this paper is organized as follows.

In \cref{sec:Preparation}, we recall central definitions and preliminary results.

In \cref{section:L1}, we discuss the ray decomposition and define $\cd^1(K,N)$ in the locally finite setting.
We show under assumptions of \cref{thm:local-to-global}, $\cd_{\loc}(K,N)$ implies $\cd^1(K,N)$.

In \cref{section:cd1tocd}, we discuss the implication $\cd^1(K,N)\Rightarrow \cd(K,N)$.

\section{Preliminaries}\label{sec:Preparation}
\subsection{Curvature-Dimension Condition}\label{section:CD}
A triple $(X, d, \mathfrak{m})$ always stands for a \emph{metric measure space} consisting of a Polish metric space equipped with the Borel $\sigma$-algebra and a locally finite Borel measure $\mathfrak{m}$ (i.e. for any $x \in X$, $\mathfrak{m} (B_r(x))<\infty$ for some $r > 0$). 
Denote $\mathcal{P}_2 (X)$ as the space of probability measures with finite variances and $\mathcal{P}_2 (X, \mathfrak{m})$ the subspace of all absolutely continuous measures w.r.t. $\mathfrak{m}$.

An \emph{optimal plan} between $\mu_0, \mu_1 \in \mathcal{P}_2(X)$ is a coupling $\pi \in\mathcal{P} (X \times X)$ minimizing the cost
\begin{equation} 
C (\omega) = \int_{X \times X} \frac{d^2 (x, y)}{2} \omega (\d x \d y) 
\end{equation}
among all $\omega \in \mathcal{P} (X \times X)$ having $\mu_0$ and $\mu_1$ as the first and second marginal.
Denote by $\opt (\mu_0, \mu_1)$ the set of all optimal plans between $\mu_0$ and $\mu_1$. 
There is a $d^2/2$-concave function $\varphi : X \rightarrow \mathbb{R}$ called a \emph{Kantorovich potential} associated to optimal plan $\pi$ satisfying
\begin{equation} 
\varphi (x) + \varphi^c (y) = \frac{d^2 (x, y)}{2}, \quad \pi-a.e. (x, y) \in X \times X
\end{equation}
where $\varphi^c$ is the \emph{conjugate potential} of $\varphi$ given by
\begin{equation}
  \varphi^c (y) \coloneqq \inf_{z \in X} \left( \frac{d^2 (y, z)}{2} - \varphi(z) \right).  \label{eq:cconcave}
\end{equation}
Define the $L^2$-Wasserstein distance between probabilities as $W_2 (\mu_0,\mu_1) \coloneqq \sqrt{C (\pi)}$ for $\pi\in\opt(\mu_0,\mu_1)$, which makes $\mathcal{P}_2 (X)$ a Polish metric space. 
Denote $\geo(X)$ the set of all constant speed geodesic $\gamma\colon [0,1]\to X$.
When endowed with the supremum distance, it is a Polish metric space.

If $(X, d)$ is geodesic, so is $(\p_2 (X) , W_2)$ (see \cite[Theorem 2.10]{Ambrosio2013}). 
Let $e_t :\geo (X) \ni \gamma \mapsto \gamma_t \in X$ be the evaluation map, and $\ell(\gamma)$ be the length.
Then for any $\mu_0, \mu_1 \in \mathcal{P}_2 (X)$, there exists a probability measure $\nu$ (referred to as an \emph{optimal dynamical plan}) on $\geo(X)$ s.t.
\begin{itemize}
  \item $(e_i)_{\#} \nu = \mu_i$, $i = 0, 1$ and $(e_0, e_1)_{\#} \nu \in
  \opt (\mu_0, \mu_1)$;
  
  \item $[0, 1] \ni t \mapsto \mu_t \coloneqq (e_t)_{\#} \nu$ is a constant
  speed geodesic in $(\mathcal{P}_2 (X), W_2)$;
  
  \item $\nu$ is concentrated on the set of \emph{Kantorovich geodesics}
  \begin{equation} G_{\varphi} \coloneqq \{\gamma \in \geo (X) : \varphi (\gamma_0) +
     \varphi^c (\gamma_1) = \ell^2 (\gamma) / 2\} .
     \end{equation}
\end{itemize}
Denote by $\optgeo (\mu_0,\mu_1)$ the set of all optimal dynamical plans.

\begin{definition}
Define the \emph{$N$-R{\'e}nyi entropy} $\mathcal{E}_N$ of any $\mu \in \mathcal{P}_2 (X, \mathfrak{m})$ by
\begin{equation}
      \mathcal{E}_N (\mu) \coloneqq \int_X \rho^{1 - 1 / N} (x)
     \d\mathfrak{m}, \quad \rho \coloneqq
  \frac{\d \mu}{\d\mathfrak{m}}. 
  \end{equation}
\end{definition}

  Given $N \in (1, \infty)$, define by the following two \emph{distortion coefficients}
    \begin{align}
  \sigma_{K, N}^{(t)} (\theta)& = \frac{\sin \left( t \theta
     \sqrt{\frac{K}{N}} \right)}{\sin \left( \theta \sqrt{\frac{K}{N}}\right)} \coloneqq 
     \left\{ \begin{array}{ll}
       \frac{\sin \left( t \theta \sqrt{\frac{K}{N}} \right)}{\sin \left(
       \theta \sqrt{\frac{K}{N}} \right)} & K > 0, \quad 0<\theta<\pi \sqrt{\frac{N}{K}}\\
       t & K = 0,\quad 0< \theta<\infty\\
       \frac{\sinh \left( t \theta \sqrt{\frac{- K}{N}} \right)}{\sinh \left(
       \theta \sqrt{\frac{- K}{N}} \right)} & K < 0, \quad 0<\theta<\infty
     \end{array} \right.,\\
     \tau_{K, N}^{(t)} (\theta)&
  \coloneqq t^{1 / N} \sigma_{K, N - 1}^{(t)} (\theta)^{1 - 1 / N}.
 \end{align}
\begin{definition}Let $(X, d, \mathfrak{m})$ be a metric-measure space.
\begin{itemize}
    \item $(X, d, \mathfrak{m})$ is said to verify
  $\cd (K, N)$, if for all $\mu_0, \mu_1 \in \mathcal{P}_2 (X,
  \mathfrak{m})$, there exists $\nu \in \optgeo (\mu_0, \mu_1)$ so that
  for all $t \in [0, 1]$, $\mu_t = (e_t)_{\#} \nu \ll \mathfrak{m}$, and for
  all $N' \geq N$:
  \begin{equation}
    \mathcal{E}_{N'} (\mu_t) \geq \int_{X \times X} \tau_{K, N^{\prime}}^{(1 - t)} (d (x_0, x_1)) \rho_0^{- 1 / N'} (x_0) +\tau_{K, N'}^{(t)} (d (x_0, x_1)) \rho_1^{- 1 / N'} (x_1) \pi (\d x_{0}, \d x_1), \label{eq:definition of CD(K,N)}
  \end{equation}
  where $\pi = (e_0, e_1)_{\#} \nu$ and $\rho_t \coloneqq \frac{\d \mu_{t}}{\d\mathfrak{m}}$.
  \item $(X, d, \mathfrak{m})$ is said to verify $\cd(K,N)$ locally, or $\cd_{\loc} (K, N)$ in short, if for any $o \in \spt (\mathfrak{m})$ one can find a neighborhood $X_o \subset X$ of $o$, so that for all $\mu_0, \mu_1 \in \mathcal{P}_2 (X,\mathfrak{m})$ supported in $X_o$, there exists $\nu\in \optgeo(\mu_0, \mu_1)$ so that $\mu_t \coloneqq (e_{t })_{\#} \nu \ll \mathfrak{m}$, and $\left(\ref{eq:definition of CD(K,N)} \right)$ holds for
  all $t \in [0, 1]$, $N' \geq N$.
  \item $(X,d,\m)$ is said to verify $\mcp(K,N)$, if for any $o \in \spt (\mathfrak{m})$ and $\mu_0\coloneqq \frac{\m\llcorner A}{\m(A)}$ given $A$ a Borel subset of $X$ with $0<\m(A)<\infty$, there exists $\nu\in\optgeo(\mu_0,\delta_o)$ s.t.
\begin{equation}\label{ineq:mcp}
    \frac{\m}{\m(A)}\geq (e_t)_{\#}(\tau^{(1-t)}_{K,N}(d(\gamma_0,\gamma_1))^N\nu(\d \gamma))\quad \forall t\in[0,1].
\end{equation}
\end{itemize}
\end{definition}

\begin{definition} A set $G\subset \geo(X)$ is \emph{non-branching} if for any $\gamma^1,\gamma^2\in G$ with $\gamma^1=\gamma^2$ on $[0,t]$ for some $t\in(0,1)$, it holds $\gamma^1=\gamma^2$ on $[0,1]$.

A space $(X,d,\m)$ is called \emph{essentially non-branching} if for all $\mu_0,\mu_1\in\p_2(X,\m)$, \emph{any} $\nu\in\optgeo(\mu_0,\mu_1)$ is concentrated on a Borel non-branching set $G\subset \geo(X)$.
  
\end{definition}

\begin{remark}\label{note:setting}
  Throughout this paper, we assume that $\spt (\mathfrak{m}) = X$ without any further specification as it will not affect the generality.
  Indeed, as discussed in \cite[Remark 6.11]{cavalletti2021globalization}, whenever $\mu_0,\mu_1\ll \m$, almost every curve in the support of $\nu\in\optgeo(\mu_0,\mu_1)$ is contained in $\spt(\m)$. So the problem on $(X,d,\m)$ is equivalent to the one on $(\spt(\m),d,\m)$.
\end{remark}

\subsection{Density Functions on Nonbranching spaces}\label{sec:nonbranchingdensity}
Cavalletti-Mondino in \cite{cavalletti2017optimal} showed that optimal maps of transports with $\mu_0\ll\m$ uniquely exist on e.n.b. $\mcp(K,N)$ spaces.
Such $\mcp$-condition is always satisfied on e.n.b. $\cd_{\loc}(K,N)$ spaces (first by \cite{cavalletti2012local} on non-branching spaces, and then on e.n.b. spaces with properties developed in \cite{cavalletti2017optimal}). 

In this subsection, $(X,d,\m)$ \textbf{always stands for an e.n.b. length m.m.s. satisfying $\cd_{\loc}(K,N)$ or $\mcp(K,N)$}.
It is well-known that any $\cd_{\loc}(K,N)$ length space is locally compact (see e.g. \cite[Lemma 6.12]{cavalletti2021globalization}), so by Hopf-Rinow Theorem, it is proper and geodesic.

\begin{proposition}[cf. {\cite{cavalletti2017optimal}}]\label{prop:optimalmaps}
For every $\mu_0,\mu_1\in\p_2(X)$ with $\mu_0\ll \m$, there exists a unique $\nu \in \optgeo (\mu_0,\mu_1)$; such $\nu$ is induced by a map (i.e. $\nu=S_{\#}\mu_0$ for $S:X\supset \mathrm{Dom}(S)\rightarrow \mathrm{Geo}(X)$) and for every $t\in(0,1)$, $(e_t)_{\#}\nu\ll\mathfrak{m}$.
\end{proposition}

\begin{lemma}[{\cite[Corollary 6.16]{cavalletti2021globalization}}]\label{lemma:injective}
Given $\mu_0,\mu_1$ as in \cref{prop:optimalmaps}, the unique optimal dynamical plan $\nu$ is concentrated on a Borel set $G\subset \geo(X)$ s.t. the evaluation map $e_{t}:G\rightarrow X$ is injective for all $t\in[0,1)$. And in particular, any Borel $H\subset G$, we have 
\[
(e_t)_{\#}(\nu\llcorner H)=({e_t}_{\#}\nu)\llcorner e_t(G)\quad \forall t\in[0,1).
\]
\end{lemma}

The following can be regarded as an expansion of the original proof in \cite{cavalletti2021globalization}.

\begin{proof}
We first assume $\mu_1\ll \m$. Recall $\nu$ is induced by a map i.e. $\nu={S_0}_{\#}\mu_0={S_1}_{\#}\mu_1$. As argued in \cite{cavalletti2021globalization}, for both $i=0,1$ we can find $X_i\subset X$ of full $\mu_i$ measure s.t. for all $x\in X_i$, there exists a unique $\gamma\in G_{\varphi}$ with $\gamma_i=x$. In particular, $\nu(S_0(X_0))=\nu(S_1(X_1))=1$. 

Take a Borel set $G\subset S_0(X_0)\cap S_1(X_1)$, still with full $\nu$-measure. 
We claim $e_t$ is injective on $G$ for all $t\in[0,1]$. By construction, $e_0$ and $e_1$ are clearly injective on $G$. 
Assume there are $\gamma,\tilde{\gamma}\in G$, $\gamma_t=\tilde{\gamma}_t$ for some $t\in(0,1)$. 
Define a curve $\eta$ by letting $\eta=\gamma$ on $[0,t]$ and $\eta=\tilde{\gamma}$ on $[t,1]$. 
By cyclic monotonicity, $\eta\in G_{\varphi}$. Since $\gamma\in S_0(X_0)$, $\eta\equiv \gamma$ on $[0,1]$ and so $\gamma_1=\tilde{\gamma}_1$. 
On the other hand, as $\gamma,\tilde{\gamma}\in S_1(X_1)$, one concludes $\gamma\equiv\tilde{\gamma}$.

For general $\mu_1\in\p_2(X)$, we prove by taking restrictions of $\nu$. 
For any $t\in[0,1)$, define 
\begin{equation}
    \mathrm{restr}^t_0:\mathrm{supp}(\nu)\rightarrow \geo(X),\quad \gamma(\cdot)\mapsto \gamma(t\cdot).
\end{equation}
\cref{prop:optimalmaps} ensures that $\mu_t\coloneqq (e_t)_{\#}\nu\ll\mathfrak{m}$, and $(\mathrm{restr}^t_0)_{\#}\nu$ is the unique optimal dynamical plan between $\mu_0$ and $\mu_t$. 
From the first step, we can find a Borel set $\tilde{G}_t$ where $(\mathrm{restr}^t_0)_{\#}\nu$ is concentrated and evaluation maps are injective over there. Then, take a sequence $t_n\nearrow 1$ and a set
\[
G\coloneqq \bigcap_{n\in\mathbb{N}}(\mathrm{restr}^{t_n}_0)^{-1}(\tilde{G}_{t_n}).
\]
One can check $\nu(G)=1$ and $e_t$ is injective on $G$ for all $t\in[0,1)$.
\end{proof}

Since $X$ is proper and any bounded subset has finite $\mathfrak{m}$-measure. Via a conditioning argument, we can extend \cite[Proposition 9.1]{cavalletti2021globalization} to infinite-volume spaces.

\begin{proposition}[Density characterization]\label{prop:densitychara}
 For any $\mu_0 \in \mathcal{P}_2 (X, \mathfrak{m})$, $\mu_1 \in
  \mathcal{P}_2 (X)$, there exists a unique $\nu \in \optgeo (\mu_0,\mu_1)$ so that for all $t \in (0, 1)$, $(e_t)_{\#} \nu \ll \mathfrak{m}$
  and
  \begin{equation}\label{eq:density of MCP} 
    \rho^{- 1 / N}_t (\gamma_t) \geq \tau^{(1 - t)}_{K, N}(d(\gamma_0,\gamma_1)) \rho^{- 1 / N}_0 (\gamma_0) 
    \quad \text{for $\nu$-a.e. } \gamma .
  \end{equation}
 It verifies $\cd (K, N)$ iff for any $\mu_0,
  \mu_1 \in \mathcal{P}_2 (X, \mathfrak{m})$, there exists a unique $\nu \in \optgeo (\mu_0, \mu_1)$ so that for all $t \in (0, 1)$, $(e_t)_{\#}\nu \ll \mathfrak{m}$ and
  \begin{equation}
    \rho^{- 1 / N}_t (\gamma_t) \geq\tau^{(1-t)}_{K,N} (d(\gamma_0,\gamma_1))\rho^{-1/N}_0(\gamma_0)+\tau^{(t)}_{K, N} (d(\gamma_0,\gamma_1)) \rho^{-1/N}_1(\gamma_1), \quad \text{for $\nu$-a.e. } \gamma . \label{eq:densityCD}
  \end{equation}
\end{proposition}

\begin{proof}[Sketch of proof]
When $\mathfrak{m}(X)<\infty$, arguing by approximation as in \cite[Proposition 9.1]{cavalletti2021globalization}, for arbitrary boundedly supported $\mu_0\in\p_2(X,\m)$ and $\mu_1\in\p_2(X)$, we have 
\begin{equation}\label{ineq:intversionofMCP}
    \mathcal{E}_N (\mu_t) \geq \int \tau^{(1 - t)}_{K, N} (d (\gamma_0,
    \gamma_1)) \rho^{- 1 / N}_0 (\gamma_0) \nu (\d
    \gamma), \quad \nu \in \optgeo (\mu_0, \mu_1) ,
  \end{equation}
 where $\mu_t=(e_t)_{\#}\nu$. 
 Here the finiteness of volume is only required for showing the upper-semicontinuity of $\mathcal{E}_N$. 
 In our case due to the choice of marginals, $(\mu_t)_t$ are confined to a fixed bounded set $U$. So redoing \cite[Lemma 4.1]{sturm2006--1} ensures that $\mathcal{E}_N$ is upper-semicontinuous w.r.t. weak convergence of measures supported inside $U$. 
  
  Now consider general $\mu_0 \in \mathcal{P}_2 (X, \mathfrak{m})$, $\mu_1 \in\mathcal{P}_2 (X)$ possibly with unbounded supports. 
  Take any compact $G\subset \geo(X)$ with $\nu (G) > 0$.
  The restricted plan $\tilde{\nu} = \frac{1}{\nu(G)} \nu \llcorner_G$ is still an optimal dynamical plan. 
  By \cref{lemma:injective}, $\tilde{\mu}_t \coloneqq (e_t)_{\#} \tilde{\nu}$ has the density $\tilde{\rho}_t = \frac{1}{\nu (G)} \rho_t\llcorner e_t (G)$, and having a uniformly bounded support. 
  So \eqref{ineq:intversionofMCP} holds for $\tilde{\nu}$, implying
  \begin{equation} \int_G \rho_t^{- 1 / N} (\gamma_t) \nu (d \gamma) \geq \int_G \tau^{(1 -
     t)}_{K, N}  (d (\gamma_0, \gamma_1)) \rho_0 (\gamma_0)^{- 1 / N} \nu (d
     \gamma) . \end{equation}
  The arbitrariness of $G$ and the inner regularity of $\nu$ yield the inequality \eqref{eq:density of MCP} for $\nu$-a.e. $\gamma$.
  
  For the second assertion on $\cd(K,N)$.
  The ``only if" part follows by applying the similar conditioning to \eqref{eq:definition of CD(K,N)}.
  The ``if'' part follows directly by
  integrating \eqref{eq:densityCD} against $\nu$. 
\end{proof}

An important consequence of the previous proposition is the following continuity of optimal dynamics, which plays a crucial role in the ray decomposition (see e.g. the proof of \cref{thm:cd_locTocd_1}). Besides, the Lipschitz-regularity of densities is a starting point of the bootstrap argument in \cite[Section 12]{cavalletti2021globalization}.

\begin{corollary}[Continuity
of Dynamics, cf. {\cite[Section 9]{cavalletti2021globalization}}] \label{cor:densityconti}
Let $\nu \in \optgeo (\mu_0, \mu_1)$ for $\mu_0,\mu_1 \in \mathcal{P}_2 (X, \mathfrak{m})$.
  \begin{enumerate}
    \item There exist versions of densities $\rho_t = \frac{\d
    \mu_t}{\d\mathfrak{m}}$, $t \in [0, 1]$, so that for $\nu$-a.e. $\gamma \in
    \geo (X)$ and all $0 \leq s < t \leq 1$:
    \begin{equation}\label{eq:density along geodesic}
      \rho_s (\gamma_s) > 0, \quad \left( \tau^{\left( \frac{s}{t}
      \right)}_{K, N} (d (\gamma_0, \gamma_t)) \right)^N \leq \frac{\rho_t
      (\gamma_t)}{\rho_s (\gamma_s)} \leq \left( \tau^{\left( \frac{1 - t}{1 -
      s} \right)}_{K, N} (d (\gamma_s, \gamma_1)) \right)^{- N} .
    \end{equation}
    In particular, for $\nu$-a.e. $\gamma$, the map $t \mapsto \rho_t
    (\gamma_t)$ is locally Lipschitz on $(0, 1)$ and upper semi-continuous at $t = 0,1$.
  \item For any compact $G \subset \geo (X)$ with $\nu (G) > 0$ s.t. \eqref{eq:density along geodesic} holds for all $\gamma \in G$ and $0 \leq s \leq t \leq 1$, we have $\mathfrak{m} (e_s (G)) > 0$ for all
    $s \in [0, 1]$ and
    \begin{equation}
      \left( \frac{1 - t}{1 - s} \right)^N e^{- d (G) (t - s)  \sqrt{(N - 1)
      K^-}} \leq \frac{\mathfrak{m} (e_t (G)) \textsf{}}{\mathfrak{m} (e_s
      (G))} \leq \left( \frac{t}{s} \right)^N e^{d (G) (t - s)  \sqrt{(N - 1)
      K^-}}, \label{eq:continuous mass of evolution set}
    \end{equation}
    where $d (G) \coloneqq \max\{ \ell (\gamma) : \gamma \in G \}$ and $K^- \coloneqq \max \{ - K, 0 \}$. In particular, the map $t\mapsto\mathfrak{m} (e_t (G))$ is locally Lipschitz on $(0, 1)$ and lower semi-continuous at $t = 0,1$.
  \end{enumerate}
\end{corollary}

\subsection{Intermediate-time Kantorovich Potentials}\label{section:potentials}

We first recall the notion of intermediate-time Kantorovich potentials.

\begin{definition}
  \label{def:int time Kp}Given a Kantorovich potential $\varphi : X
  \rightarrow \mathbb{R}$, the \emph{intermediate-time Kantorovich potential}
  $\varphi_t$ at time $t \in [0, 1]$ is defined by $\varphi_0 = \varphi$, $\varphi_1 = -\varphi^c$ and 
  \begin{equation} \varphi_t (x) \coloneqq - \underset{y \in X}{\inf} \left[ \frac{d^2 (x,
     y)}{2 t} - \varphi (y) \right]. \end{equation}
\end{definition}

Denote the domain of the dynamics and its section through $x$ as:
\begin{align} 
D ({G}_{\varphi}) &\coloneqq \{(x, t) \in X \times (0, 1) : \exists\gamma \in G_{\varphi}, x = \gamma_t \}, \\
G_{\varphi}(x)&\coloneqq \{t\in (0,1):(x,t)\in D(G_{\varphi})\}.
\end{align}

Based on \cref{lemma:injective}, when $(X,d,\m)$ is e.n.b., for simplicity we will assume $e_{t}:G_{\varphi}\rightarrow \R$ is injective for all $t\in[0,1]$ as otherwise it suffices to restrict $\nu$ to some Borel $G\subset G_{\varphi}$.
Then the length function is defined by
\[
\ell:D(G_{\varphi})\ni(x,t)\mapsto \ell(e_t^{-1}(x))\coloneqq \mathrm{Length}(e_t^{-1}(x))
\]
and we also use the notation $\ell_t(\cdot)\coloneqq \ell(\cdot, t)$ on $e_t (G_{\varphi})$ for every $t$.

\begin{definition}
  Given a Kantorovich potential $\varphi : X \rightarrow \mathbb{R}$
  and $s, t \in (0, 1)$, define the \emph{$t$-propagated $s$-Kantorovich potential}
  $\Phi_s^t$ on $e_t (G_{\varphi})$ by
  \begin{equation} 
  \Phi_s^t\coloneqq\varphi_s \circ e_s \circ e_t^{- 1}=\varphi_t+(t-s)\frac{\ell^2_t}{2}. 
  \end{equation}
\end{definition}

For every fixed $s\in(0,1)$, according to the value of $\varphi_s$, $G_{\varphi}$ can be partitioned into closed levels
\begin{equation}
   G_{\varphi} = \sqcup_{a_s \in \im (\varphi_s \circ e_s)}G_{\varphi, a_s}, \quad  G_{\varphi, a_s} \coloneqq (\varphi_s \circ e_s)^{-1} (a_s) = \{\gamma \in G_{\varphi} : \varphi_s (\gamma_s) = a_s \} 
\end{equation}
which further leads to a partition of $e_t(G_{\varphi})$ (any $t\in(0,1)$) via $\Phi_s^t$ by
\begin{equation} 
e_t (G_{\varphi}) = \sqcup_{a_s \in \im (\varphi_s \circ e_s)} e_t(G_{\varphi, a_s}), \quad e_t (G_{\varphi, a_s})\coloneqq (\Phi_s^t)^{- 1} (a_s) =\{\gamma_t : \varphi_s (\gamma_s) = a_s \}. 
\end{equation}

\begin{lemma}[Continuity of potentials, cf. {\cite[Section 3]{Ambrosio_2013}} and {\cite[Theorem 3.11, Proposition 4.4]{cavalletti2021globalization}}]\label{lemma:continity}
The function $X \times (0, 1) \ni (x, t) \mapsto \varphi_t (x)$ is locally Lipschitz. 
The length function $\ell$ is continuous on $D(G_{\varphi})$.
For any $x \in X$ and $s\in (0, 1)$, functions ${G}_{\varphi}(x) \ni t \mapsto \ell_t (x)$ and  ${G}_{\varphi}(x) \ni t \mapsto \Phi^t_s (x)$ are locally Lipschitz.
\end{lemma}

%\section{Localization and Conclusion}\label{sec:Conclusion}
%\input{Conclusion.tex}
\section{$L^1$-disintegration}\label{section:L1}
\subsection{Disintegration Theorem}\label{section:Disintegration}

Proofs of assertions in this subsection can be found in {\cite[Appendix
A]{Bianchini-Caravenna}}.
Let $(X, \mathfrak{X}, \mathfrak{m})$, $(Q, \mathcal{Q}, \mathfrak{q})$ be
measure spaces. A \emph{disintegration} of $\mathfrak{m}$ over $\mathfrak{q}$ is a family of measures $(\mathfrak{m}_q)_{q \in Q}$ on $X$ s.t. for every $E \in
\mathfrak{X}$, the map $q \mapsto \mathfrak{m}_q (E)$ is
$\mathfrak{q}$-measurable and $\mathfrak{m} (E) = \int_Q \mathfrak{m}_q (E)
\mathfrak{q} (\d q)$. By {\cite[Proposition 452F]{Fremlin}}, for any
$\mathfrak{m}$-measurable $\xi : X \rightarrow \mathbb{R}$, we have
\begin{equation}
  \int_Q \int_X \xi (x) \mathfrak{m}_q  (\d x) \mathfrak{q} (\d q) = \int_X \xi
  (x) \mathfrak{m} (\d x), \label{eq:definedisintbyf}
\end{equation}
provided $\int \xi (x) \mathfrak{m} (\d x)$ is well-defined in $\mathbb{R} \cup
\{\pm \infty\}$.

Given measurable $f : (X, \mathfrak{X}) \rightarrow (Q, \mathcal{Q})$, a disintegration $(\mathfrak{m}_q)_{q \in Q}$ of $\mathfrak{m}$ over $\mathfrak{q}$ is called \emph{consistent} with $f$ if for each $I\in\mathcal{Q}$,
  \begin{equation}
    \mathfrak{m} (E \cap f^{- 1} (I)) = \int_I \mathfrak{m}_q (E) \mathfrak{q}
    (\d q) . \label{eq:defineconsisdisint}
  \end{equation}
And $(\mathfrak{m}_q)_{q \in Q}$ is called \emph{strongly consistent} with $f$, if for $\mathfrak{q}$-a.e. $q \in Q$, $\mathfrak{m}_q$ is concentrated on $f^{- 1}(\{q\})$. 
Clearly, strong consistency implies consistency.
\begin{remark}[Uniqueness of Disintegration]
  \label{rmk:uniqueofdisinte}If $\mathfrak{X}$ is countably generated with
  a $\sigma$-finite measure $\mathfrak{m}$, and disintegrations $(\mathfrak{m}_q)$,
  $(\tilde{\mathfrak{m}}_q)$ of $\mathfrak{m}$ over
  $\mathfrak{q}$ are consistent with $f$, then $\mathfrak{m}_q =
  \tilde{\mathfrak{m}}_q$ for $\mathfrak{q}$-a.e. $q \in Q$, or in short,
  consistent disintegrations are {\emph{$\mathfrak{q}$-unique}}.
  
  Indeed, by {\cite[Proposition 3.3]{preston2008notes}}, there is a countable
  subalgebra $\{B_n \in \mathfrak{X}, n \in \mathbb{N}\}$ generating
  $\mathfrak{X}$. 
  After putting $E = B_n$ into \eqref{eq:defineconsisdisint}, we know up to a $\mathfrak{q}$-negligible set $N \subset Q$, $\mathfrak{m}_q
  (B_n) = \tilde{\mathfrak{m}}_q (B_n)$ for all $n$ and $q$. 
  So when $\mathfrak{m}$ is finite, with Dynkin's theorem, $\mathfrak{m}_q =\tilde{\mathfrak{m}}_q$ for all $q \in Q \setminus N$. 
  For the case where $\m $ is $\sigma$-finite, we can repeat the previous argument on any subset $E$ of finite $\mathfrak{m}$-measure to show $\mathfrak{m}_q\llcorner E=\tilde{\mathfrak{m}}_q\llcorner E$ for a.e. $q$.
  The argument is complete after taking an exhausting sequence $E_n$ of $X$.
  
  In particular, strongly consistent disintegrations of a locally finite measure are $\mathfrak{q}$-unique.
\end{remark}

If $X$ has a partition $\Pi = \{X_q \}_{q \in Q}$, define $\mathfrak{Q}: X
\rightarrow Q$ by mapping each point in $X_q$ to $q$. 
Endowed with the quotient $\sigma$-algebra $\mathcal{Q}$ and the quotient measure $\mathfrak{q}=\mathfrak{Q}_{\#} \mathfrak{m}$, $(Q, \mathcal{Q},\mathfrak{q})$ is a measure space.

\begin{definition}
  A \emph{cross section} of a partition $\Pi$ is a subset $S$ of $X$ so that $S \cap A$ is a singleton
  for each $A \in \Pi$. 
  A \emph{section} is a map $\mathfrak{S}: X \rightarrow X$ such that for each $x \in X$, the image of $[x]$ under $\mathfrak{S}$ is a singleton in $[x]$, where $[x]$ is the equivalence class of $x$ under $\Pi$.
  
  A subset $S_{\mathfrak{m}}$ is called an \emph{$\mathfrak{m}$-section} if there exists a \emph{Borel} set $\Gamma \subset X$ s.t. $\m(X\setminus \Gamma)=0$ and the partition $\Pi_{\Gamma} = \{X_q \cap \Gamma\}_{q \in Q}$ has $S_{\mathfrak{m}}$ as a cross
  section.
\end{definition} 

\begin{theorem}[Disintegration Theorem]
  \label{thm:disintegration}Assume $(X, \mathfrak{X}, \mathfrak{p})$ is a
  countably generated probability space, having a partition $\Pi = \{X_q \}_{q
  \in Q}$. Let $\mathfrak{Q}: X \rightarrow Q$ and $(Q, \mathcal{Q}, \mathfrak{q})$ be the quotient map and quotient space resp. There exists a unique disintegration
  $q \mapsto \mathfrak{p}_q \in \mathcal{P} (X)$ of $\mathfrak{p}$ over
  $\mathfrak{q}$ consistent with $\mathfrak{Q}$.
 Moreover, this disintegration is strongly consistent with $\mathfrak{Q}$ iff there exists a Borel $\mathfrak{p}$-section $S_{\mathfrak{p}}\subset Q$ s.t. the quotient $\sigma$-algebra $\mathcal{Q} \cap S_{\mathfrak{p}}$ contains $\mathcal{B} (S_{\mathfrak{p}})$.
\end{theorem}

\begin{remark}[Disintegration over level sets]\label{rmk:disintegrationLevelset}
If $(X, d)$ is Polish
    and the partition $\Pi$ given as level sets of a continuous function $\mathfrak{Q}:X \rightarrow \mathbb{R}$, then, by \cite[{Theorem}
    5.4.3]{srivastava2013course}, $\Pi$ admits a Borel cross-section $S$ and
    Borel section map $\mathfrak{S}$. In particular, there is a unique
    disintegration of $\mathfrak{p}$ strongly consistent with $\mathfrak{S}$.
  
\end{remark}

\subsection{Transport Ray and $\cd^1(K,N)$}\label{section:Raydecomposition}

For any $1$-Lipschitz function $u:(X,d)\rightarrow \R$, define the
\emph{transport relation} $R_u$ and the \emph{transport set} $\mathcal{T}_u$ as
\begin{equation} R_u \coloneqq \{(x, y) \in X \times X : |u (x) - u (y) | = d (x, y)
   \}, \qquad \mathcal{T}_u \coloneqq P_1  (R_u \setminus \{x = y\}),
\end{equation}
where $P_i$ is the projection onto the $i$-th component. Denote $R_u (x) \coloneqq
\{y \in X : (x, y) \in R_u \}$ as the section of $R_u$ through $x$ in the
first coordinate.

Notice $R_u$ is not necessarily an equivalence relation as the transitivity may be violated. To remedy this, define the \emph{non-branched transport set} by removing those branched points:
\begin{equation} 
\mathcal{T}^{b}_u \coloneqq \{x \in \mathcal{T}_u : \forall z, w
   \in R_u (x), (z, w) \in R_u \} 
\end{equation}
and hence the corresponding non-branched transport relation
\begin{equation} R^b_u \coloneqq R_u \cap (\mathcal{T}^{b}_u \times
   \mathcal{T}^{b}_u) . \end{equation}
\begin{remark}\label{rmk:measurablityofTS}
  We refer to {\cite{Bianchini, cavalletti2013monge}} and \cite[Section 7]{cavalletti2021globalization} for following statements:
  \begin{itemize}
    \item When $(X, d)$ is proper, $\mathcal{T}_u$ is $\sigma$-compact, and
    $\mathcal{T}_u^b$, $R^{b}_u$ are Borel;
    
    \item $R^b_u$ is an equivalence relation on $\mathcal{T}^b_u$ which
    induces a partition $\sqcup_x R^b_u (x)$ of $\mathcal{T}^b_u$;
    
    \item When $(X, d)$ is geodesic, for any $x \in \mathcal{T}^b_u$, $R_u
    (x)$ is a single (unparameterized) geodesic of positive length, so that $(R_u (x), d)$ is isometric to a closed interval in $(\mathbb{R}, | \cdot |)$ and $(R^b_u (x), d)$ is a subinterval.
  \end{itemize}
\end{remark}
We call $R$ a \emph{transport ray} if $(R,d)$ is isometric to a closed interval in $(\R,|\cdot|)$ of positive length and it is maximal under the partial order $\leq_u$, where $x\leq_u y$ if $u(x)-u(y)=d(x,y)$.
%%%%%%
\iffalse
{\color{red}\begin{remark}\label{remark:transportset}
  When the cost function is given by $c(x, y) = d (x, y)$, a Kantorovich potential $\varphi$ for any couple $\mu_0$, $\mu_1$ is $1$-Lipschitz, and for any optimal dynamics $\nu$, the set
  \begin{equation} 
  e_{[0, 1]} (\spt (\nu)) \coloneqq \{\gamma_t, \gamma \in \spt(\nu), t \in [0, 1]\} 
  \end{equation}
  is contained in $\mathcal{T}_{\varphi}$ up to some $\mathfrak{m}$-negligible set.
\end{remark}}
\fi 
%%%%%%%%%
\begin{definition}
  Given a continuous function $\phi : (X, d) \rightarrow \mathbb{R}$ so that
  $\{ \phi = 0 \}\neq \emptyset$, define the \emph{signed distance function} (from zero-level set of $\phi$) as
  \begin{equation}
  d_{\phi}: X \rightarrow \mathbb{R}, \quad d_{\phi} (x) \coloneqq\mathrm{dist} (x, \{\phi = 0\}) \mathrm{sign}(\phi) . 
 \end{equation}
\end{definition}

When $(X,d)$ is a length space, any signed distance function $d_{\phi}$ is $1$-Lipschitz (see \cite[Lemma 8.4]{cavalletti2021globalization}).
If further $\m(X)<\infty$, \cref{thm:disintegration} gives a disintegration of $\m$ on $\mathcal{T}^b_{d_{\phi}}$ w.r.t. the partition by $R^b_{d_{\phi}}$, which leads to the $\cd^1$-condition introduced in \cite{cavalletti2021globalization}.
We modify this condition by relaxing conditional measures to be only locally finite, instead of probabilities.
\begin{definition}\label{def:cd1}
    A m.m.s. $(X,d,\mathfrak{m})$ with $\spt(\m)=X$ satisfies $\cd^1(K,N)$ if for any $1$-Lipschitz signed distance function $u=d_{\phi}$, with the associated partition $\{R^b_u(q)\}_{q\in Q}$ of $\mathcal{T}^b_u$ by ray decomposition, there exist a probability space $(Q,\mathcal{Q},\q)$ and a $\q$-unique disintegration $\mathfrak{m}\llcorner\mathcal{T}_u =\int_Q\mathfrak{m}_q\mathfrak{q}(\d q)$ on $
    \{\overline{R^b_u(q)}\}_{q\in Q}$ s.t.
  \begin{enumerate}
    \item\label{itm:cd1_1}  
     $Q$ is a section of the above partition so that $Q\supseteq \bar{Q}\in \mathcal{B}(\mathcal{T}^b_u)$ with $\bar{Q}$ an $\m$-section with $\m$-measurable quotient map and $\mathcal{Q}\supseteq \mathcal{B}(\bar Q)$;
    \item\label{itm:cd1_2} for $\mathfrak{q}$-a.e. $q \in Q$, $\overline{R^b_u(q)}=R_u(q)$ as a transport ray;
      \item\label{itm:cd1_3} for $\q$-a.e. $q\in Q$, $\mathfrak{m}_q$ is non-null, supported on $\overline{R^b_u(q)}$; 
      \item\label{itm:cd1_4} for $\mathfrak{q}$-a.e. $q \in Q$, $(\overline{R^b_u(q)}, d, \mathfrak{m}_q)$
    is a one-dimensional $\cd (K, N)$ m.m.s.;
      \item\label{itm:locallyfinite} for every bounded subset $K \subset X$, there exists $C_K \in (0, \infty)$ s.t.
   \begin{equation}\label{ineq:locallyfinite} 
       \mathfrak{m}_q (K) \leq C_K, \quad \text{for $\mathfrak{q}$-a.e. $q$}.
       \end{equation}
  \end{enumerate}
\end{definition}
\begin{remark}\label{rmk:onCD^1_0}
     The reference measure $\m$ on any $\cd^1(K,N)$ space must be locally finite, simply by \eqref{ineq:locallyfinite} and taking $u=d(\cdot,o)$ for some $o\in X$. 
     And by \cref{thm:disintegration}, the disintegration is strongly consistent with the quotient measure because of \eqref{itm:cd1_1}.
\end{remark}

\begin{remark}\label{rmk:onCD^1}
    For any $u$ and disintegration from \cref{def:cd1}, $\m(\mathcal{T}_u)=\m(\mathcal{T}^b_u)$.
    Indeed, denoting by $\tilde Q\subset Q$ the set of $q$ that \eqref{itm:cd1_2}-\eqref{itm:cd1_4} hold, then
    \begin{align}
        \m(\mathcal{T}_u)&=\int_Q \m_q(\mathcal{T}_u)\q(\d q)=\int_{\tilde Q} \m_q(\overline{R^b_u(q)})\q(\d q)\\
        &=\int_{\tilde Q} \m_q( R^b_u(q))\q(\d q)=\int_{\tilde Q} \m_q( \mathcal{T}^b_u)\q(\d q)=\m(\mathcal{T}^b_u),
    \end{align}
    where we have used the fact that a measure carrying $\cd(K,N)$ does not charge points.
\end{remark}
%%%%%%%%%%
\subsection{$\cd_{\loc}(K,N)$ implies $\cd^1(K,N)$}\label{section:cd_loctocd1}
\begin{theorem}\label{thm:cd_locTocd_1}
  Let $(X, d, \mathfrak{m})$ be an e.n.b. $\cd_{\loc}(K, N)$ length m.m.s. such that $\m$ is locally finite with full-support, and $u:(X, d)\rightarrow\mathbb{R}$ be a $1$-Lipschitz function. 
  Then there exists a disintegration of $\m\llcorner \mathcal{T}_u$ satisfying \eqref{itm:cd1_1}-\eqref{itm:locallyfinite} of \cref{def:cd1}.
  In particular, under these assumptions, $\cd_{\loc}(K,N)$ implies $\cd^1(K,N)$.
\end{theorem}

Such disintegration, also called ray/needle decomposition, is extensively studied in e.g. \cite{Bianchini,cavalletti2013monge,CavallettiMondinoInven} under the assumption $\m(X)=1$.
However in our case, $\mathcal{T}^b_u$ could be unbounded with infinite volume, so we can not directly apply \cref{thm:disintegration}. 
Therefore we normalize the measure by adding a weight function following the approach in \cite{Cavalletti-Mondino2020}.
After such re-weighting, $\cd$-information can be passed to rays exactly as in the finite-volume case.

\begin{proof}
    As every $\cd_{\loc} (K, N)$ geodesic m.m.s. is proper, \cref{rmk:measurablityofTS} applies. 
    From \cite[Proposition 4.5]{cavalletti2013monge} (together with the comments above \cite[Corollary 7.3]{cavalletti2021globalization}), $\mathfrak{m}(\mathcal{T}_u\setminus \mathcal{T}^b_u)=0$.
    Hence it suffices to disintegrate $\m\llcorner \mathcal{T}^b_u$ w.r.t. the partition $\mathcal{T}^b_u=\sqcup_q R^b_u (q)$.
    
\textbf{Normalize $\m$ to apply the disintegration theorem.} Without loss of generality we assume $\mathfrak{m}(\mathcal{T}^b_u) = \infty$.
Then, for any fixed $x_0 \in X$, we can find an increasing sequence $(r_n)_{n\geq 1}$ of positive numbers, so that 
\begin{equation}\label{eq:annuli} \mathcal{T}^{b, n}_u \coloneqq \left\{\begin{array}{ll}
     \mathcal{T}^b_u \cap \{x \in X : r_n \leq d (x, x_0) < r_{n+1} \} & , n
     \in \mathbb{N}_+\\
     \mathcal{T}^b_u \cap \{x \in X : d (x, x_0) < r_1\} & , n = 0
   \end{array}\right. \end{equation}
has positively finite $\m$-measure for each $n\geq 0$.
Define $f$ by
\begin{equation} f (x) = \sum_{n \in \mathbb{N}} (2^{n + 1} \mathfrak{m}(\mathcal{T}^{b,
   n}_u))^{- 1} \mathds{1}_{\mathcal{T}^{b, n}_u} (x).
   \end{equation}
Clearly,
\begin{equation}\label{eq:normalziedfctf}
  \inf_{K\cap \mathcal{T}^b_u} f > 0, \text{ for any compact } K \subset X ; \quad
  \int_{\mathcal{T}^b_u} f (x) \mathfrak{m} (\d x) = 1.
\end{equation}
Hence, $\mathfrak{n} \coloneqq
f\mathfrak{m} \llcorner \mathcal{T}^b_u$ is a probability measure and \cref{thm:disintegration} can be applied to $\mathfrak{n}$.

\textbf{On the strong consistency.} First, from \cite[Proposition 4.4]{Bianchini}\footnote{The existence of such section map depends only on (1): selection theorem of partitions into closed sets and (2): continuity and local compactness of geodesics, but not on the finiteness of $\m(X)$.}, there exists an $\mathfrak{m}$-measurable section $\mathfrak{Q}: \mathcal{T}^b_u \rightarrow \mathcal{T}^b_u$ associated to the partition $\{R^b_u (q)\}_{q \in Q}$. 
From now on, we fix $Q$ as the image of $\mathfrak{Q}$ and endow $Q$ with $\sigma$-algebra $\mathcal{Q}\coloneqq \mathfrak{Q}_{\#}\mathcal{B}(\mathcal{T}^b_u)$. 
Take $\mathfrak{q}=\mathfrak{Q}_{\#} \mathfrak{n} \textsf{}$ to be the quotient
probability on $Q$.
By the $\mathfrak{m}$-measurability of $\mathfrak{Q}$, $\mathfrak{q}$ is Borel on $\mathcal{T}^b_u$. Thus $\mathfrak{q}$ is inner regular and we can find a $\sigma$-compact set $S \subset Q$, $\mathfrak{q}(Q \setminus S) = 0$.
Define a section $\mathfrak{S} \coloneqq \mathfrak{Q} \llcorner(\mathfrak{Q}^{- 1} (S))$ where $\mathfrak{Q}^{- 1} (S)$ has full $\mathfrak{n}$-measure. Then
  \begin{equation} 
  \mathrm{graph} (\mathfrak{S}) = \{(x, s) \in \mathcal{T}^b_u\times S:(x,s)\in R_u \} 
  \end{equation}
  is Borel, implying that $\mathfrak{Q}^{- 1} (S) = P_1 (\mathrm{graph}(\mathfrak{S}))$ is analytic and $\mathfrak{S}$ is Borel measurable by \cite[Theorem 4.5.2]{srivastava2013course}.
  That is to say, $S$ is a Borel $\mathfrak{n}$-section with Borel measurable section $\mathfrak{S}$ and hence $\mathcal{Q}\supset \mathcal{B}(S)$.
  \cref{thm:disintegration} applies to conclude that $q \mapsto \mathfrak{n}_q$ is the $\mathfrak{q}$-unique disintegration of $\mathfrak{n}$ strongly consistent with $\mathfrak{Q}$.
  In particular, \eqref{itm:cd1_1} is verified as $\mathfrak{n}$ and $\m\llcorner\mathcal{T}^b_u$ sharing same measurable and null sets.

Back to $\mathfrak{m} \llcorner \mathcal{T}^b_u$, owing to the everywhere positivity of $f$ on $\mathcal{T}^b_u$, we have
\begin{equation}\label{eq:disintafterweight}
\mathfrak{m} \llcorner \mathcal{T}^b_u = \int_Q \mathfrak{n}_q /
   f\mathfrak{q} (\d q).
   \end{equation}
Define $\mathfrak{m}_q \coloneqq\mathfrak{n}_q /f$. 
As measurability (w.r.t. $q \in Q$) is guaranteed, $q\mapsto \mathfrak{m}_q$ gives the unique disintegration of $\m\llcorner \mathcal{T}^b_u$ strongly consistent with $\mathfrak{Q}$ (recall \cref{rmk:uniqueofdisinte}).
From \eqref{eq:normalziedfctf}, $\m_q$ is uniformly-locally finite as \eqref{ineq:locallyfinite}.
Further, we can repeat \cite[Theorem 7.10]{cavalletti2021globalization}(which mainly needs \cref{prop:optimalmaps} but not finiteness of $\m(X)$, as it is proved by contradiction and localization) for \eqref{eq:disintafterweight} to show that for $\mathfrak{q}$-a.e. $q \in Q$, $R_u (q) = \overline{R^b_u (q)}$ and $\spt (\mathfrak{m}_q) = R_u (q)$. 

\textbf{Localize $\cd(K,N)$ to transport rays.} Let $S$ be the $\sigma$-compact cross section in the previous step.
 Define the \emph{ray map} $g : \mathrm{Dom}(g) \subset S \times \mathbb{R} \rightarrow \mathcal{T}^b_u$ via
  \begin{equation}\label{eq:raymap}
  \mathrm{graph} (g) \coloneqq \{ (q, t, x) \in S \times \mathbb{R} \times\mathcal{T}^b_u : (q, x) \in R_u, u (x) - u (q) = t \} .
 \end{equation}
\cref{rmk:measurablityofTS} ensures that each $x \in \mathcal{T}^b_u$ uniquely corresponds a pair $(\mathfrak{Q} (x), d) \in S \times \mathbb{R}$, with $d = u (x) - u (\mathfrak{Q} (x))$ and $| d | = d (x, \mathfrak{Q}(x))$. 
Hence $g$ is well-defined, bijective and Borel measurable because of its
Borel graph. 
For any $q \in S$, $I_q \coloneqq \mathrm{Dom} (g (q, \cdot))$ is an
interval in $\mathbb{R}$, and $I_q \ni t \mapsto g (q, t) \in R^b_u (q)$ is an
isometry, meaning $\mathcal{H}^1 \llcorner \{ R^b_u (q) \} = g (q, \cdot)_{\#}
(\mathcal{L}^1 \llcorner I_q)$.

It remains to show that for $\mathfrak{q}$-a.e. $q \in Q$, $\mathfrak{m}_q \ll g(q,\cdot)_{\#}\mathcal{L}^1$ and for those $q$, by denoting $\mathfrak{m}_q = g(q, \cdot)_{\#}(h_q \cdot \mathcal{L}^1 \llcorner I_q)$, $h_q$ is a $\cd (K, N)$ density\footnote{See Appendix for the definition of $\cd(K,N)$ densities.}. 
Such regularity problem for conditional measures can be solved by combining arguments in \cite[Theorem 5.7]{Bianchini} and \cite[Theorem 4.2]{CavallettiMondinoInven}. 
We refer to the appendix for more detailed demonstrations.
\end{proof}
Observe that \eqref{eq:disintafterweight} depends on the chosen normalization of the reference measure.
However, this affects the disintegration only by a constant factor on each ray.
Namely, given two disintegrations with a weight function $f$ and $g$ respectively
\begin{equation}
    \m\llcorner \mathcal{T}^b_u=\int_Q \mathfrak{n}^f_q/f\q^f(\d q)=\int_Q \mathfrak{n}^g_q/g\q^g(\d q)
\end{equation}
as constructed in the proof, where the quotient space $(Q,\mathcal{Q})$ does not rely on normalizations.
By the positivity of weight functions, $\q^f$ and $\q^g$ are mutually absolutely continuous.
Hence the essential uniqueness of consistent disintegration yields an equality between $(\m^f_q)_q$ and $(\m_q^g)_q$.

Nevertheless, existence result of the disintegration is sufficient for our purpose.
\section{From $\cd^1(K,N)$ to $\cd(K,N)$}\label{section:cd1tocd}
This section is devoted to the following main theorem, which together with \cref{thm:cd_locTocd_1} concludes the local-to-global property of $\cd(K,N)$.
\begin{theorem}\label{thm:cd1tocd}
Let $(X,d,\m)$ be an e.n.b. m.m.s. with $\m$ locally finite having full support. If it holds $\cd^1(K,N)$, then it holds $\cd(K,N)$.
\end{theorem}
It turns out that the approach developed in \cite{cavalletti2021globalization} is powerful enough to work on locally finite spaces with very mild modifications once the $\cd^1$-condition is given by \cref{def:cd1}.
In subsequent sections, we sketch the proof with the absence of the finiteness of $\m$, following closely \cite{cavalletti2021globalization}, highlighting necessary modifications.
Note that the following part is by no means self-contained, so a parallel reading on the paper \cite{cavalletti2021globalization} is recommended for readers looking for details.

\subsection{$\cd^1(K,N)$ implies $\mcp(K,N)$}
We begin with recovering the $\mcp$-condition.
\begin{proposition}\label{prop:cd1tomcp}
If a m.m.s. $(X,d,\m)$ verifies $\cd^1(K,N)$, then it verifies $\mcp(K,N)$.
\end{proposition}
\begin{proof}
By definition, we need to show that for any $o\in X$ and $\mu_0\coloneqq \frac{\m\llcorner A}{\m(A)}$, there exists $\nu\in\optgeo(\mu_0,\delta_o)$ such that \eqref{ineq:mcp} is satisfied, where $A\subset X$ is an arbitrary Borel set with $0<\m(A)<\infty$.
We can further assume $A$ to be bounded by \cite[Remark 5.1]{rajala-jfa}.

Choosing $u=d(\cdot,o)$, the $\cd^1$-condition provides a disintegration of $\m$ on $\mathcal{T}_u=X$ s.t. for $\q$-a.e. $q\in Q$, $(R_u(q),d,\m_q)$ verifies $\cd(K,N)$, and in particular $\mcp(K,N)$.
Based on the uniform-local finiteness \eqref{ineq:locallyfinite}, the function $Q\ni q\mapsto \m_q(A)$ is $\q$-measurable and almost everywhere finite.
For all $q$ in 
\begin{equation}
    \bar Q\coloneqq\{q\in Q:\m_q(A)\in(0,\infty),\spt(\m_q)=R_u(q)=\overline{R_u^b(q)}\},
\end{equation}
define $\mu^q_0\coloneqq \frac{\m_q\llcorner A}{\m_q(A)}$.
By the maximality of $R_u(q)$, $o\in \spt(\m_q)$ and there exists a unique $\nu^q\in\optgeo(\mu^q_0, \delta_o)$ for $q\in \bar Q$.
Take $\nu=\int_{\bar Q}\nu^q\frac{\m_q(A)}{\m(A)}\q(\d q)$ and all curves in its support are contained in a common bounded subset of $X$.
Then going in lines of the proof of \cite[Proposition 8.9]{cavalletti2021globalization} validates that $\nu$ is a required optimal dynamical plan from $\mu_0$ to $\delta_o$.
\end{proof}
As a consequence, all statements in \cref{sec:nonbranchingdensity} now hold on e.n.b. $\cd^1(K,N)$ spaces and the underlying metric space must be Polish, proper and geodesic.

Let $\mu_0$ and $\mu_1$ be two arbitrary elements in $\mathcal{P}_2(X,\m)$ and $\nu\in\optgeo(\mu_0,\mu_1)$. 
Fix a Kantorovich potential $\varphi$ of the quadratic optimal transport from $\mu_0$ to $\mu_1$ and denote by $(\varphi_t)_{t \in [0, 1]}$ the family of intermediate-time Kantorovich potentials.

By \cref{prop:densitychara}, it is sufficient to show that the density $\rho_t$ of $\mu_t\coloneqq (e_t)_{\#}\nu$ w.r.t. $\m$ satisfies the distortion inequality
\begin{equation}\label{eq:CDKNintro}
  \rho^{- 1 / N}_t (\gamma_t) \geq \tau^{(1 - t)}_{K, N} (\ell
  (\gamma)) \rho^{- 1 / N}_0 (\gamma_0) +
  \tau^{(t)}_{K, N} (\ell (\gamma)) \rho^{- 1 / N}_1
  (\gamma_1) 
\end{equation}
for $\nu$-a.e. $\gamma\in \geo(X)$.

For this aim, we will localize the whole problem to transport paths, by coupling different disintegrations.
Since we only care the almost-everywhere statement of \eqref{eq:CDKNintro}, by \cref{sec:nonbranchingdensity}, we can work under the following convenient convention.
\begin{convention}\label{convention:goodcurves}
In the sequel, we restrict ourselves to a Borel subset of Kantorovich geodesics of full $\nu$-measure, still denoted by $G_{\varphi}$ with a slight abuse of notation, such that
\begin{enumerate}
    \item $e_t$ is injective on $G_{\varphi}$ for all $t\in[0,1]$;
    \item $(\rho_t)_t$ can be chosen that statements in (1) of \cref{cor:densityconti} hold for each $\gamma\in G_{\varphi}$.
\end{enumerate}
\end{convention}

\subsection{$L^2$-decomposition of transports.}\label{section:L2disint}
Based on discussions in \cref{section:potentials}, for fixed $s,t\in[0,1]$, we have two families of partitions given by level sets of continuous functions as follows
\begin{equation}\label{eq:partitions}
G_{\varphi}=\sqcup_{a_s\in\R}G_{\varphi,a_s},\quad 
    e_t (G_{\varphi}) = \sqcup_{a_s\in\R} e_t(G_{\varphi, a_s})
\end{equation}
 where $G_{\varphi,a_s}\coloneqq \{\gamma\in G_{\varphi}: \varphi_s(\gamma_s)=a_s\}$.

Replace $G_{\varphi}$ by any compact subset $G$ with $\nu(G)>0$ and by \cref{rmk:disintegrationLevelset}, there exist disintegrations of finite measures $\nu\llcorner G$ and $\m\llcorner e_t(G)$ strongly consistent with partitions \eqref{eq:partitions} respectively.
 Notice that all arguments in \cite[Section 10.2]{cavalletti2021globalization} can be repeated without any change so quotient measures are absolutely continuous to the one-dimensional Lebesgue measure $\mathcal{L}^1$ for both disintegrations induced. 
  More precisely, we can find $(\nu_{a_s})$ and $(\m^t_{a_s})$ concentrated on $G_{a_s}(\coloneqq G\cap G_{\varphi,a_s})$ and $e_t (G_{a_s})$ respectively so that
  \begin{equation}\label{eq:IntroL2disint}
  \nu = \int \nu_{a_s} \mathcal{L}^1 (\d a_s), \quad \mathfrak{m}
     \llcorner e_t (G) = \int \mathfrak{m}^t_{a_s} \mathcal{L}^1 (\d a_s).
  \end{equation}
The two families of conditional measures in \eqref{eq:IntroL2disint} are comparable under the relation 
\begin{align}
  \mu_t \llcorner e_t (G) & =(e_t)_{\#}(\nu \llcorner G)=\int_{\varphi_s (e_s (G))} (e_t)_{\#} \nu_{a_s} \mathcal{L}^1 (\d a_s)\\
  & = \rho_t\mathfrak{m}\llcorner e_t(G)=\int_{\varphi_s (e_s (G))} \rho_t \cdot \mathfrak{m}^t_{a_s}\mathcal{L}^1 (\d a_s).\label{eq:compareL2dis}
\end{align}
By \cref{rmk:uniqueofdisinte}, for $\mathcal{L}^1$-a.e. $a_s \in \varphi_s (e_s (G))$, $\rho_t \cdot \mathfrak{m}_{a_s}^t =(e_t)_{\#} \nu_{a_s}$.
%%%%%%%%%%%%%%%%%%%%
\subsection{$L^1$-decomposition of $\mathfrak{m}$ via needle decomposition.} 
For any $s \in (0, 1)$ and $a_s \in \im (\varphi_s \circ e_s)$, denote $u \coloneqq d_{\varphi_s-a_s}$ as the signed
distance function from $\{ \varphi_s = a_s \}$. By \cite[Lemma 10.3]{cavalletti2021globalization}, for every $\gamma \in G_{\varphi, a_s}$ and $0 \leq r \leq t \leq 1$, $(\gamma_r, \gamma_t) \in R_u$. 
In particular, $e_{[0, 1]} (G_{\varphi, a_s}) \subset \mathcal{T}_u$. 

Again, when we restrict the $L^1$-disintegration to a compact subset, and with the uniform boundedness of conditional measures given by \eqref{ineq:locallyfinite}, a repetition of \cite[Propositon 10.4]{cavalletti2021globalization} can be performed as follows.

\begin{proposition}\label{prop:L^1disonG}
For any compact subset $G\subset G^+_{\varphi}$ with positive measure, $s\in(0,1)$ and $a_s\in \varphi_s(e_s(G))$, we have the following disintegration: 
\begin{equation}
  \label{disinte formula 3} \mathfrak{m} \llcorner e_{[0, 1]}
  (G_{a_s}) = \int_{[0, 1]} \mathfrak{m}^{a_s}_t \mathcal{L}^1 (\d t),
\end{equation}
where 
\begin{equation}\label{eq:transitioncondim}
  \mathfrak{m}^{a_s}_t = g^{a_s} (\cdot, t)_{\#}  (h^{a_s}_{\cdot} (t)\cdot\mathfrak{m}^{a_s}_s)
\end{equation}
so that
\begin{enumerate}
    \item $g^{a_s}:e_s(G_{a_s})\times [0,1]\rightarrow X$ is Borel measurable, mapping $(\beta,t)$ to $e_t(e^{-1}_s(\beta))$;
    \item  $(0,1)\ni t\mapsto \m^{a_s}_t$ is continuous under weak convergence, and for each $t$, $\m^{a_s}_t$ is concentrated on $e_t(G_{a_s})$;
    \item for $\m^{a_s}_s$-a.e. $\beta$, $h^{a_s}_{\beta}$ is a continuous $\cd(\ell^2_s(\beta)K,N)$ density on $(0,1)$ satisfying $h^{a_s}_{\beta}(s)=1$;
    \item\label{item:lemma3.5(1)} there exists a constant $C$ depending only on $K, N$ and $\max\{\ell (\gamma):\gamma \in G\}$,
    \begin{equation}\label{ineq:m^a_s_t}
    \| \mathfrak{m}^{a_s}_t \| \leq C\mathfrak{m} (e_{[0, 1]} (G_{a_s})),\quad \forall t\in(0,1).
    \end{equation}
    %\item\label{item:lemma3.5(2)}the following limit holds in the weak topology (dual of $C_b(X)$) 
    %\begin{equation}
    %    \mathfrak{m}^{a_s}_t =\lim_{\epsilon \rightarrow 0}\frac{1}{2 \epsilon} \mathfrak{m} \llcorner e_{(t - \epsilon, t +\epsilon)} (G_{a_s}).
   % \end{equation}
\end{enumerate}
\end{proposition} 
\begin{proof}

\textbf{Restrict $L^1$-disintegration to curves in $G_{a_s}$.} 
By definition, one has a probability measure $\hat\q^{a_s}$ and a disintegration 
\begin{equation}\label{eq:proofL1dis1}
    \m\llcorner \mathcal{T}_{u}=\int_Q \hat\m^{a_s}_q\hat\q^{a_s}(\d q).
    \end{equation}
Since $e_{[0, 1]}(G_{a_s}) \subset \mathcal{T}_u$, we can restrict \eqref{eq:proofL1dis1} to $e_{[0, 1]}(G_{a_s})$ so that
\begin{equation}
    \m\llcorner e_{[0, 1]}(G_{a_s})=\int_Q \hat\m^{a_s}_q\llcorner e_{[0, 1]}(G_{a_s})\hat\q^{a_s}(\d q).
\end{equation}
If we denote
\begin{align}
    G^{1}_{a_s}\coloneqq \{\gamma\in G_{a_s}:\mathcal{T}^b_u\cap e_{[0,1]}(\gamma)\neq \emptyset\},\quad Q^1\coloneqq \{q\in Q: R^b_u(q)\cap e_{[0,1]}(G_{a_s})\neq \emptyset\},
\end{align}
then following exactly same arguments in Part 1-3 of the proof of \cite[proposition 10.4]{cavalletti2021globalization} on the ray decomposition and measurability we know $Q^1$ is $\hat\q^{a_s}$-measurable, $G^1_{a_s}$ is analytic, and there exists a Borel isomorphism $\eta : (Q^1, \mathcal{B}(Q^1)) \rightarrow (G_{a_s}^1,\mathcal{B}(G_{a_s}^1))$, mapping $q$ to $\gamma^q$.

On the other hand, there exists $\tilde Q\subset Q$ of full $\hat\q^{a_s}$-measure s.t. for each $q\in \tilde Q$, $\hat\m^{a_s}_q$ is non-null, supported on $R_u(q)=\overline{R^b_u(q)}$ and $(R_u(q),d,\hat\m^{a_s}_q)$ verifies $\cd(K,N)$.
Since each $\gamma^q\in G_{a_s}$ is contained in $R_u(q)$, $\{\gamma_q\}_{q\in \tilde Q}$ have disjoint interiors.
By \eqref{ineq:locallyfinite} and the fact that a non-null measure carrying $\cd(K,N)$ gives positive mass to open sets, we have 
\begin{equation}
   0<\hat\m^{a_s}_q(e_{[0,1]}(G_{a_s}))=\hat\m^{a_s}_q(e_{[0,1]}(\gamma^q))<C_G,\quad \forall q\in Q^1\cap \tilde Q,
\end{equation}
for some constant $C_{G}>0$.
Therefore, summarizing above discussions with \cref{rmk:onCD^1} gives
\begin{align}
     \mathfrak{m} \llcorner e_{[0, 1]} (G_{a_s})& = \int_{Q^1\cap \tilde Q} \hat\m^{a_s}_q
     \llcorner \{ \mathcal{T}^b_u \cap e_{[0, 1]} (G_{a_s}) \} \hat\q^{a_s}(\d q)\\
     &= \int_{Q^1\cap \tilde Q} \frac{\hat\m^{a_s}_q \llcorner e_{[0, 1]}(\gamma^q)}{\hat\m^{a_s}_q (e_{[0, 1]} (\gamma^q))} \hat\m^{a_s}_q(e_{[0, 1]} (\gamma^q)) \hat\q^{a_s}(\d q).\label{eq:proofL1dis2}
\end{align}
Denoting $\bar\m^{a_s}_q\coloneqq \frac{\hat\m^{a_s}_q \llcorner
  e_{[0, 1]}(\gamma^q)}{\hat\m^{a_s}_q (e_{[0, 1]} (\gamma^q))}$, and
  $\bar\q^{a_s} =\hat\m^{a_s}_q (e_{[0, 1]} (\gamma^q)) \hat\q^{a_s}\llcorner Q^1$, \eqref{eq:proofL1dis2} can be rewritten as
  \begin{equation}\label{eq:proofL1dis3}
    \mathfrak{m} \llcorner e_{[0, 1]} (G_{a_s})=\int_{Q^1\cap \tilde Q}
    \bar{\mathfrak{m}}^{a_s}_q  \bar{\mathfrak{q}}^{a_s} (\d q).\label{eq:proofL1dis3}
  \end{equation}
  \textbf{Change the variable and conditional measures.}
Pushing-forward via the Borel measurable bijection $e_{s}\circ \eta:Q^1\to e_s(G^1_{a_s})$ induces a space $(e_s(G^1_{a_s}),\mathcal{S},\check\q^{a_s})$, with $\mathcal{S}\coloneqq (e_s\circ \eta)_{\#}(\mathcal{Q}\cap Q^1)$ and $\check\q^{a_s}=(e_s\circ\eta)_{\#}\bar\q^{a_s}$.
  Correspondingly, \eqref{eq:proofL1dis3} can be expressed on the new measurable space:
  \begin{equation}\label{eq:proofL1dis4}
    \mathfrak{m} \llcorner e_{[0, 1]} (G_{a_s}) = \int_{e_s (G_{a_s})}\m^{a_s}_{\beta} \check\q^{a_s}  (\d \beta),
  \end{equation}
  where $\mathfrak{m}^{a_s}_{\beta} =\bar\m^{a_s}_{(e_s\circ\eta)^{-1}(\beta)}$ has unit mass and $\check\q^{a_s}$ is concentrated on $e_s\circ\eta(Q^1\cap\tilde Q)$ since $\bar\q^{a_s}$ and $\hat\q^{a_s}\llcorner Q^1$ are mutually absolutely continuous.

 With the new cross section $e_s (G_{a_s})$, we define a ray map $g^{a_s}$ as in \eqref{eq:raymap} but now with the time variable fixed on $[0, 1]$:
  \[ g^{a_s} : e_s (G_{a_s}) \times [0, 1] \rightarrow X, \quad (\beta, t)
     \mapsto e_t (e_s^{- 1} (\beta)) . \]
 Clearly, $g^{a_s}$ is Borel measurable. 
 For any $\beta \in e_s (G_{a_s})$, $t \mapsto g^{a_s} (\beta, t)$
  is an isometry between $([0, 1], | \cdot |)$ and $(\gamma^{\beta}\coloneqq e_s^{- 1} (\beta), d /\ell_s (\beta))$. 
  By assumption, for $\check\q^{a_s}$-a.e. $\beta$, $(\gamma^{\beta}, d, \mathfrak{m}^{a_s}_{\beta})$ verifies $\cd(K, N)$.
  After rescaling the metric, $(\gamma^{\beta}, d /\ell_s (\beta), \mathfrak{m}^{a_s}_{\beta})$ verifying $\cd(\ell_s^2(\beta) K, N)$.
 For those $\beta$, there exists a continuous function $\check h^{a_s}_{\beta}$ as a $\cd(\ell_s^2 (\beta) K, N)$ probability density on $(0,1)$ s.t.
  \begin{equation}
    \mathfrak{m}^{a_s}_{\beta} = g^{a_s} (\beta, \cdot)_{\#} (\check h^{a_s}_{\beta}
    \cdot \mathcal{L}^1 \llcorner [0, 1])\label{eq:conditional
    probabilityanddensity}
  \end{equation}
  and 
      \begin{equation}\label{proofL1dis4}
     \mathfrak{m} \llcorner e_{[0, 1]} (G_{a_s}) = \int_{e_s(G_{a_s})}  g^{a_s} (\beta, \cdot)_{\#} (\check h^{a_s}_{\beta}
    \cdot \mathcal{L}^1 \llcorner [0, 1])\check\q^{a_s}(\d \beta).
  \end{equation}
 % Take $\bar Q$ from \eqref{itm:cd1_1} of \cref{def:cd1}. We can find a subset $\mathcal{Q}\ni\tilde Q^1\subset Q^1$ so that $\hat\check\q^{a_s}(Q^1\setminus \tilde Q^1)=0$ and $\mathcal{Q}\supseteq \mathcal{B}(\tilde Q^1)$.Since $\bar\check\q^{a_s}$ and $\hat\check\q^{a_s}\llcorner Q^1$ are mutually absolutely continuous, $\bar\check\q^{a_s}(Q^1\setminus \tilde Q^1)=0$. Hence $\check\q^{a_s}$ is a Borel measure, concentrated on $\tilde S$.
 \textbf{Reformulate the disintegration on $[0,1]$.}
   The item \eqref{itm:cd1_1} of \cref{def:cd1} allows us to repeat the step 8 of the proof of \cite[Proposition 10.4] {cavalletti2021globalization} to obtain the $\check\q^{a_s}\otimes \mathcal{L}^1$-measurability of $e_s(G_{a_s})\times [0,1]\ni(\beta, t)\mapsto \check h^{a_s}_{\beta}(t)$, where we also follow the convention that $\check h^{a_s}_{\beta}$ vanishes at endpoints.
  By Fubini, we can exchange the order of \eqref{eq:proofL1dis4} s.t. \eqref{disinte formula 3} is achieved with 
  \begin{equation}\label{eq:proofL1dis5}
      \mathfrak{m}^{a_s}_t = g^{a_s} (\cdot, t)_{\#}  (\check h^{a_s}_{\cdot} (t)\cdot\check\q^{a_s}).
  \end{equation}
 The map $\beta\mapsto \check h^{a_s}_{\beta}(s)$ is $\check\q^{a_s}$-measurable and for $\check\q^{a_s}$-a.e. $\beta$, $\check h^{a_s}_{\beta}$ is strictly positive on $(0,1)$. 
 Let $h^{a_s}_{\beta}\coloneqq\frac{\check h^{a_s}_{\beta}}{\check h^{a_s}_{\beta}(s)}$ for those $\beta$ and $\q^{a_s}\coloneqq\check h^{a_s}_{\beta}(s)\cdot\check\q^{a_s}$.
 Now $h^{a_s}_{\beta}(s)=1$ for $\q^{a_s}$-a.e. $\beta$, and $\check\q^{a_s}$, $\q^{a_s}$ are mutually absolutely continuous (both of them are finite measures) sharing same measurable and null sets.
 And $g^{a_s}(\cdot,t)_{\#}( h^{a_s}_{\cdot} (t)\cdot\q^{a_s})$ equals to $\m^{a_s}_t$ still.
 Hence, $\m^{a_s}_s=\q^{a_s}$ and the translation relation \eqref{eq:transitioncondim} is satisfied.
 
The continuity of $t\mapsto \m^{a_s}_t$ follows from the continuity of $t\mapsto h^{a_s}_{\beta}(t),g^{a_s}(\beta,t)$.
 Finally, by \cite[Lemma A.8]{cavalletti2021globalization}, probability densities $\check h^{a_{s}}_{\beta}(t)$ are bounded uniformly for $\beta,t$.
The uniform volume bound \eqref{ineq:m^a_s_t} of $\m^{a_s}_t$ is given by \eqref{eq:proofL1dis5} and the finiteness of 
$\check\q^{a_s}$:
 \begin{align}
     \check \q^{a_s}(e_{[0,1]}(G_{a_s})=\bar\q^{a_s}(Q^1)=\int_{Q^1}\hat\m^{a_s}_q (e_{[0, 1]} (\gamma^q)) \hat\q^{a_s}(\d q)=\m(e_{[0,1]}(G_{a_s})).&\qedhere
 \end{align}
\end{proof}
%%%%%%%%%%%%%%%%%%%%%%%
\subsection{Comparison between conditional measures.}
This section is to recover the comparison between $L^2$ and $L^1$ disintegrations based on \cite[Section 11]{cavalletti2021globalization}.

Recall that the $t$-propagated $s$-Kantorovich potential defined on $D({G_{\varphi}})$, by $\Phi^t_s (x) \coloneqq\varphi_t (x) + \frac{t - s}{2} \ell_t^2 (x)$, is jointly continuous and locally Lipschitz on $t$. 
The following differential properties will be crucial in the comparison argument.
Moreover, they are statements of metric spaces without any reference measure.

\begin{lemma}[cf. {\cite[Proposition 4.4]{cavalletti2021globalization}} ]
  \label{prop:derivative of Phi}Fix any $s \in (0, 1)$.
  \begin{enumerate} 
    \item For any $x\in X$, $t \mapsto \Phi^t_s (x)$ is differentiable iff $t \mapsto \ell^2_t (x)$ is differentiable on $G_{\varphi}(x)$ or $t=s\in G_{\varphi}(x)$, with derivatives
    \begin{equation}
      \partial_t \Phi^t_s (x) = \ell^2_t (x) + (t - s)  \frac{\partial_t
      \ell^2_t (x)}{2}, \quad \partial_t |_{t = s} \partial_t \Phi^t_s (x) =
      \ell^2_s (x) . \label{eq:derivative of Phi}
    \end{equation}
    \item For all $(x, t)\in D(G_{\varphi})$,
    \begin{align}
       \min\big\{\frac{s}{t},\frac{1-s}{1-t}+&\frac{t-s}{t(1-t)}\big\}\ell^2_t(x)\leq \liminf_{G_{\varphi}(x)\ni \tau\to t} \frac{\Phi^{\tau}_s(x)-\Phi^t_s(x)}{\tau-t}\\
       &\leq \limsup_{G_{\varphi}(x)\ni \tau\to t} \frac{\Phi^{\tau}_s(x)-\Phi^t_s(x)}{\tau-t}\leq \max\big\{\frac{s}{t},\frac{1-s}{1-t}+\frac{t-s}{t(1-t)}\big\}\ell^2_t(x).\label{eq:}
    \end{align}
  \end{enumerate}
\end{lemma}

%\begin{corollary}\label{cor:differentiable of Phi}For $\mathcal{L}^1$-a.e. t including $t =s$, and $\mathfrak{m}$-a.e. $x \in e_t (G^+_{\varphi})$, $\partial_{\tau}\Phi^{\tau}_s |_{\tau = t}(x)$ exists and is positive. Moreover,if $G \subset G^+_{\varphi}$ is a good subset, then we have the pointwise estimate:
%\begin{equation}
%    \frac{1}{2} \left( 1 + \frac{s}{t} \right)\inf_{\gamma \in G} \ell^2(\gamma) \leq \underline{\partial_t} \Phi^t_s \leq \overline{\partial_t}\Phi^t_s \leq C_t \sup_{\gamma \in G} \ell^2 (\gamma), \quad\text{on$e_{(0, 1)} (G)$}, \label{eq:derivative bound of Phi over good set}
%  \end{equation}
%  where $C_t > 0$ depends only on t and is locally bounded w.r.t. t.
%\end{corollary}

\begin{proposition}
\label{prop:comparion of measure} Let G be a compact subset of $G_{\varphi}^+$ with $\nu(G)>0$. 
For any $s \in (0, 1)$, $\mathcal{L}^1$-a.e. $t \in (0, 1)$ including $t = s$ and $\mathcal{L}^1$-a.e. $a_s \in \varphi_s (e_s (G))$, we have
  \begin{equation}\label{eq:comparison}
  \mathfrak{m}^{a_s}_t = \partial_t \Phi^t_s \cdot \mathfrak{m}^t_{a_s},
  \end{equation}
  where $\partial_t \Phi^t_s(x)$ exists and is positive for $\m^{t}_{a_s}$-a.e. $x$.
\end{proposition}

\begin{proof}[Sketch of proof]
 By \cref{prop:derivative of Phi}, for any $x \in X$, $\partial_t \Phi^t_s (x)$ exists for $\mathcal{L}^1$-a.e. $t \in  G_{\varphi}(x)$ including $t = s$, and wherever differentiable, $\partial_t \Phi^t_s (x)>0$. 
 Then the statement on differentiability can be concluded by applying Fubini to the set $\{(x, t) : \exists \gamma \in G_{\varphi}, \gamma_t = x \} \subset X \times(0, 1)$ to rephrase the exceptional set of differentiation, together with the disintegration $\m\llcorner e_t(G)=\int\mathfrak{m}^t_{a_s}\L^1(\d a_s)$.

Now we start to show the equivalence.
The main idea is a ``sum-up'' of $\mathfrak{m}^{a_s}_t$ for all $a_s \in\varphi_s (e_s (G))$ to recover $\mathfrak{m} \llcorner e_t (G)$ (which has a disintegration $\int\mathfrak{m}^t_{a_s}\L^1(\d a_s)$) and so to connect the two families of conditional measures.
  
For $t_0 \in \mathbb{R}$ and $x_0 \in X$, define
\begin{equation}
  1^1_{t_0} : X \ni x \mapsto (t_0, x) \in \mathbb{R} \times X, \quad
  1^2_{x_0} : \mathbb{R} \ni t \mapsto (t, x_0) \in \mathbb{R} \times X.
  \label{eq:lift operator}
\end{equation}
Recall that $\m^{a_s}_t$ is supported inside a common compact set (e.g. $e_{[0,1]}(G)$) for all $t$ and $a_s$, having uniformly bounded mass by \eqref{ineq:m^a_s_t}.
Using the following variant of \eqref{disinte formula 3}
\begin{equation}
\mathfrak{m} \llcorner e_{(t-\epsilon,t+\epsilon)}(G_{a_s})=\int_{(t-\epsilon,t+\epsilon)} \mathfrak{m}^{a_s}_{\tau}\mathcal{L}^1 (\d \tau),
\end{equation}
and the continuity of $\tau\mapsto \m^{a_s}_{\tau}$, we have 
\begin{equation}
     \int_{\varphi_s (e_s (G))} (1^1_{a_s})_{\#} \mathfrak{m}^{a_s}_t\mathcal{L}^1 (\d a_s) =\underset{\epsilon \rightarrow 0}{\lim} \int_{\varphi_s (e_s (G))} \frac{1}{2 \epsilon}
    (1^1_{a_s})_{\#} \mathfrak{m} \llcorner e_{(t - \epsilon, t + \epsilon)} (G_{a_s}) \mathcal{L}^1 (\d a_s),
\end{equation}
under the weak topology. Manipulating the right-hand side via Fubini and functions $(\Phi^t_s)_t$ as in the proof of \cite[Theorem 11.3]{cavalletti2021globalization}, one gets
\begin{equation}
    \int_{\varphi_s (e_s (G))} (1^1_{a_s})_{\#} \mathfrak{m}^{a_s}_t\mathcal{L}^1 (\d a_s)=\underset{\epsilon \rightarrow 0}{\lim} \int_{e_t (G)}
    \frac{1}{2 \epsilon} (1^2_x)_{\#} (\mathcal{L}^1 \llcorner\{\Phi^{\tau}_s (x) : \tau \in (t - \epsilon, t + \epsilon) \cap G(x)\}) \mathfrak{m} (\d x),
\end{equation}
where we only need the uniform boundedness of mass of measures
  \begin{equation} \label{eq:one-dimension measure on R}
    \frac{1}{2 \epsilon} \mathcal{L}^1 \llcorner \{\Phi^{\tau}_s (x) : \tau\in (t -\epsilon, t + \epsilon) \cap G(x)\}
  \end{equation}
  for all $\epsilon$, which is ensured by \cref{prop:derivative of Phi} and the compactness of $G$. 
  Besides, since $t\mapsto \Phi^t_s (x)$ is a strictly monotone Lipschitz function on $G(x)\cap (t-\epsilon,t+\epsilon)$ (with a uniform Lipschitz bound for $x\in e_t(G)$), one-dimensional measures in \eqref{eq:one-dimension measure on R} converge to $\partial_t \Phi^t_s(x) \delta_{\Phi^t_s (x)}$ when $\epsilon\to 0$, for  $\mathcal{L}^1$-a.e. $t \in (0, 1)$ and $\mathfrak{m}$-a.e. $x \in e_t(G)$.

  As a result, for $\L^1$-a.e. $t$ and $a_s$,
  \begin{equation}
       \int_{\varphi_s (e_s (G))} (1^1_{a_s})_{\#} \mathfrak{m}^{a_s}_t\mathcal{L}^1 (\d a_s) = \int_{e_t (G)} (1^2_x)_{\#} (\partial_t \Phi^t_s (x)\delta_{\Phi^t_s (x)})\m(\d x).
  \end{equation}
Testing the above equality by $1 \otimes f \in C_b (\mathbb{R} \times X)$ with the disintegration \eqref{eq:IntroL2disint} of $\m\llcorner e_t(G)$ implies 
  \begin{equation}\label{eq:proofequiv1}
      \int_{\varphi_s (e_s (G))} \mathfrak{m}^{a_s}_t \mathcal{L}^1(\d a_s) = \int_{\varphi_s (e_s (G))} \partial_t \Phi^t_s \cdot\m^t_{a_s} \mathcal{L}^1 (\d a_s).
  \end{equation}
  Actually, disintegrations on both sides of \eqref{eq:proofequiv1} are strongly consistent on $\{e_{t}(G_{a_s})\}_{a_s}$ of $e_t(G)$, and hence \eqref{eq:comparison}.
  The assertion for $t=s$ can be proven exactly the same as \cite[Theorem 11.3]{cavalletti2021globalization}, since the underlying space verifies $\mcp(K,N)$ by \cref{prop:cd1tomcp} and $\ell(\gamma)$ is uniformly bounded by the compactness of $G$.
\end{proof}

\subsection{Proof of the main theorem}
Once all disintegrations and the comparison are produced, a so-called change-of-variable formula (cf. Equation (11.10) in \cite{cavalletti2021globalization}) can be derived.
The remaining part after that, though highly technical, not related to the finiteness of $\m$, will follow naturally.
Here we outline the proof of the main theorem in the locally finite case, closely following Section 11.2, 12 and 13.1 of \cite{cavalletti2021globalization}.

\begin{proof}[Proof of \cref{thm:cd1tocd}] 
\textbf{Deriving the change-of-variable formula.} 
First, consider any $G$, as a compact subset of $G^+_{\varphi}$ with positive $\nu$-measure. 
Fix $s\in(0,1)$.
As mentioned in the end of \cref{section:L2disint}, for every $t\in(0,1)$, $\mathcal{L}^1$-a.e. $a_s \in \varphi_s (e_s (G))$, $\rho_t \cdot \mathfrak{m}_{a_s}^t =(e_t)_{\#} \nu_{a_s}$.
By evaluating both of them to $e_t(H)$ for an arbitrary Borel $H\subset G$, we have 
  \begin{equation}
    \int_{e_t (H)} \rho_t (x) \cdot \mathfrak{m}^t_{a_s} (\d x)=\nu_{a_s} (H).
  \end{equation}
In the above integral, replacing $\m^t_{a_s}$ by $(\partial_t
  \Phi^t_s)^{- 1} \mathfrak{m}^{a_s}_t$ using \cref{prop:comparion of measure}, and combining the translation formula \eqref{eq:transitioncondim}, we have
  \begin{align}\label{eq:integrand}
      \int_{e_s (H)} \underbrace{(\rho_t \cdot(\partial_t \Phi^t_s)^{- 1}) \circ g^{a_s}
    (\beta, t) h^{a_s}_{\beta} (t)}_{\coloneqq f_{t}(\beta)}\mathfrak{m}^{a_s}_s (\d \beta)=\nu_{a_s} (H),
  \end{align}
  for $\mathcal{L}^1$-a.e. $t \in (0, 1)$ including $t=s$ and $a_s \in \varphi_s (e_s (G))$.
  Denote by $f_{t}(\beta)$ the integrand in \eqref{eq:integrand}.
 From the arbitrariness of $H$, there is a subset $T\subset [0,1]$ of full measure s.t. for each $t\in T$, $f_{t}=f_{s}$ for $\m^{a_s}_{s}$-a.e. $\beta$ (due to the continuity of $\rho_t(\cdot)$, $h^{a_s}_{\beta}(\cdot)$ and $g^{a_s}(\beta,\cdot)$, and the fact that $\partial_t\Phi^t_s(x)$ converges to $\ell^2_s(x)$ when $t\to s$ by \cref{prop:derivative of Phi}).
  Recall that $h^{a_s}_{\beta}(s)=1$ for $\m^{a_s}_{s}$-a.e. $\beta$ and $g^{a_s}(\cdot, s)=id$.
  Hence, for $\L^1$-a.e. $t$,
  \begin{equation}\label{eq:a.e.integrand}
      f_{t}(\beta)=(\rho_t \cdot(\partial_t \Phi^t_s)^{- 1}) \circ g^{a_s}
    (\beta, t) h^{a_s}_{\beta} (t)=f_{s}(\beta)=\rho_{s}(\beta)/\ell^2_s(\beta),
  \end{equation}
  for $\m^{a_s}_{s}$-a.e. $\beta\in e_s (G)$.
  Again by \cref{prop:comparion of measure},
  $\mathfrak{m}^{a_s}_s$ and $\mathfrak{m}_{a_s}^s$ are mutually absolutely continuous, so \eqref{eq:a.e.integrand}
  holds for $\mathfrak{m}_{a_s}^s$-a.e. $\beta $ as well. 
  Further, the validity of \eqref{eq:a.e.integrand} for almost each $a_s$ indicates, after recovering $\m\llcorner e_s(G)$ by disintegration $\int\m^s_{a_s} \mathcal{L}^1(\d a_s)$, that \eqref{eq:a.e.integrand} holds for $\m$-a.e. $\beta = \gamma_s$ with $\gamma \in G$.

  In conclusion, after changing the variable $\beta$ to $\gamma_s$, for $\nu$-a.e. $\gamma \in G$, and $\mathcal{L}^1$-a.e. $t \in (0, 1)$, we have
\begin{equation}\label{eq:cov}
    \frac{\rho_s(\gamma_s)}{\rho_t(\gamma_t)}=\frac{ h_{\gamma_s}^{\varphi_s (\gamma_s)} (t)}{\partial_{\tau}|_{\tau=t}\Phi_s^{\tau}(\gamma_t)/\ell^2(\gamma)}.
\end{equation}
 Recall from the construction in \cref{prop:L^1disonG} that, $\check h_{\beta}^{a_s} (t)$ is uniquely defined as the continuous density of $\hat\m^{a_s}_{q}$ (given by the $L^1$-disintegration \eqref{eq:proofL1dis1} of $\mathcal{T}_u$) after conditioning it on $e_{[0, 1]} (\gamma^q)$ and pulling it back to the interval $[0, 1]$ via the ray map $g^{a_s}$ ( which can be defined on the whole $e_s(G_{\varphi}^+)\times [0,1]$ by \cref{convention:goodcurves}).
 In particular, $h^{a_s}_{\beta}$ and hence \eqref{eq:cov} does not depend on the choice of $G$.
 Then by the inner regularity of $\nu$, the validity of \eqref{eq:cov} holds for 
 $\nu$-a.e. $\gamma\in G^{+}_{\varphi}$.
 
 \textbf{``L-Y" decomposition of the density along $\gamma\in G_{\varphi}^+$.}
 We show that along each $\gamma$ satisfying \eqref{eq:cov}, the density admits a decomposition $\rho_t(\gamma_t)^{-1}=L(t)Y(t)$, where $L$ is concave and $Y$ is a $\cd(\ell^2(\gamma)K,N)$ density on $(0,1)$.
 
 All steps in the proof of \cite[Theorem 12.3]{cavalletti2021globalization} can be repeated since it is only a matter of one-dimensional analysis on $[0,1]$.
 Once we check that condition (C) is satisfied in the statement of \cite[Theorem 12.3]{cavalletti2021globalization} (the validity of (A) and (B) is clear by \cref{convention:goodcurves} and \cref{prop:L^1disonG}).
 Indeed, the condition is reduced to an estimate of the 3-rd order derivative of $t\mapsto\varphi_t(\gamma_t)$, where no difference occurs between finite and locally finite spaces.

 Afterwards, an application of H\"older's inequality (cf. \cite[Theorem 13.2]{cavalletti2021globalization}) to the ``L-Y" decomposition, with the upper semi-continuity of $t\mapsto \rho_t(\gamma_t)$ at $t=0,1$ from \cref{convention:goodcurves}, yields the desired inequality \eqref{eq:CDKNintro}.

 \textbf{On null-geodesics.}
 Denote by $G_{\varphi}^0$ the set of all curves in $G_{\varphi}$ with zero length and $X_0\coloneqq e_{[0,1]}(G^0_{\varphi})$.
 By \cite[Corollary 9.8]{cavalletti2021globalization}, as a consequence of \cref{cor:densityconti}, $\mu_t\llcorner X_0=\mu_0\llcorner X_0$ for all $t\in [0,1]$.
 As a result, same to the step 0 of \cite[Theorem 11.4]{cavalletti2021globalization} we can always redefine $\rho_t\llcorner X_0\coloneqq \rho_0\llcorner X_0$ so that \eqref{eq:cov} holds automatically over $\gamma\in G^0_{\varphi}$ and $t\mapsto\rho_t(\gamma_t)$ will not be affected for all $\gamma\in G_{\varphi}^+$.
 \end{proof}

\appendix
\section{Proof of Ray Decomposition}\label{sec:app}
\begin{definition} 
A non-negative function $h$ on an interval $I\subset \R$ is called a $\cd(K,N)$ density if for all $x_0,
  x_1 \in I$ and $t \in [0, 1]$:
\begin{equation}\label{eq:CDKNdensity}
    h (t x_1 + (1 - t) x_0)^{\frac{1}{N - 1}} \geq \sigma_{K, N - 1}^{(t)}
    (|x_0 - x_1 |) h (x_1)^{\frac{1}{N - 1}} + \sigma_{K, N - 1}^{(1 - t)}(|x_0 - x_1 |) h (x_0)^{\frac{1}{N - 1}}. 
\end{equation}
\end{definition}
The name comes from the fact that a $1$-dimensional m.m.s. $(I,|\cdot|,\mu)$ verifies $\cd(K,N)$ if and only if $\mu\ll\L^1$ and the density $h=\d\mu/\d\L^1$ has a version being a $\cd(K,N)$ density (see \cite[Theorem A.2]{cavalletti2021globalization}). Moreover, if $h\in C^2_{\loc}(I)$, then $h$ is a $\cd(K,N)$ density if and only if
\begin{equation}\label{eq:AppendCD}
    \frac{((\log h)')^2}{N - 1} +
  (\log h)'' \leq - K.
\end{equation}
  
We call a property on $I \subset \mathbb{R}$ \emph{local} if
once it holds on an interval $I_x$ of any point $x \in I$, then it holds
globally on $I$. In particular, being positive, locally
Lipschitz or a $\cd (K, N)$ density are all local properties in $\mathbb{R}$ (see {\cite[Section 5]{cavalletti2012local}} for the local-to-global
property of $\cd (K, N)$ densities).

\begin{proof}[Completion of the Proof of \cref{thm:cd_locTocd_1}]
  Via the ray map $g$ introduced in \eqref{eq:raymap}, $\m\llcorner \mathcal{T}^b_u$ can be reformed as a measure on $S \times \mathbb{R}$:
  \begin{equation} 
    (g^{- 1})_{\#} (\m \llcorner \mathcal{T}^b_u) = (g^{- 1})_{\#}
     \m \llcorner (\mathfrak{Q}^{- 1} (S)) = \int_S (g^{- 1} (q,
     \cdot))_{\#} \m_q \mathfrak{q} (\d q), 
     \end{equation}
  where the second equality is guaranteed by the strong consistency of disintegration $q\mapsto \m_q$. As $(g^{- 1} (q,
  \cdot))_{\#} \m_q$ is locally-finite on $\mathbb{R}$ from \eqref{eq:normalziedfctf}, Lebesgue's decomposition gives
  \begin{equation} (g^{- 1} (q, \cdot))_{\#} \m_q = h_q \mathcal{L}^1 + \omega_q,
     \quad \omega_q \perp \mathcal{L}^1 . \end{equation}
  
  ({\romannumeral 1}) It suffices to show for $\mathfrak{q}$-a.e. $q \in S_{a, b}$, $\m_q$ verifies \cref{thm:cd_locTocd_1} on $[a, b]$, where 
  \[
  S_{a, b} \coloneqq \{ q \in S :[a,b]\subset I_q, a,b\in\mathbb{Q} \}.
  \]
  Such $S_{a,b}$ is always $\q$-measurable, since 
  \[
  S_{a,b}= P_1(\{(x,y,z)\in S^3:(x,y),(x,z)\in R_u,u(y)-u(x)\geq b, u(z)-u(x)\leq a\}).
  \]
  Let $S_{a,b}^{*}$ be the set of all $q$ in $S_{a,b}$ s.t. \cref{thm:cd_locTocd_1} is violated somewhere in $[a,b]$. 
  As all statements are local, the set $Q^{*}\coloneqq\{q\in S:\m_q\text{ does not verify \cref{thm:cd_locTocd_1}}\text{ on }I_q\}$ is contained in $\cup_{a,b\in \mathbb{Q}}S^{*}_{a,b}$.  
  
  It can be reduced to show each $S^*_{a,b}$ is negligible and hence for the time being, we assume $S$ a bounded subset of $S_{a, b}$. 
  For simplicity, we directly assume $\m$ a measure on $S \times\mathbb{R}$ to avoid writing $g$ all the time. 
  
  ({\romannumeral 2}) Prove that $\m_q \ll\mathcal{L}^1$. 
  If otherwise, there exists a bounded set $A \subset \mathcal{T}^b_u \subset S
  \times \mathbb{R}$, $\m (A) > 0$ but for $\mathfrak{q}$-a.e. $q
  \in S$
  \begin{equation} \label{eq:vanishingmassonray}
    \mathcal{L}^1 (A \cap (\{ q \}\times\mathbb{R})) = 0.
  \end{equation}
  Take $\nu$ to be the unique optimal dynamical plan transporting $\mu_0 \coloneqq
  \m (A)^{- 1} \m \llcorner A$ onto $S \times \{ a \}$ along vertical rays
  $R^b_u (q) = \{ q \} \times I_q$. Denote $A_t \coloneqq e_t (\spt
  (\nu))$,
  \begin{equation}
    A_t = \{ (q, (1-t)\tau + ta) : (q, \tau) \in A \} .
    \label{eq:expressAt}
  \end{equation}
  \cref{cor:densityconti} ensures $\m (A_t) > 0$ for each
  $t \in [0, 1)$, so a contradiction follows:
   \begin{eqnarray*}
    0 & < & \int_0^{1 / 2} \m (A_t) \d t =\m \otimes
    \mathcal{L}^1 (\{ (q, \tau, t) \in S \times \mathbb{R} \times [0, 1 / 2] :
    (q, \tau) \in A_t \})\\
    & = & \int_{S \times \mathbb{R}} \mathcal{L}^1 (\{ t \in [0, 1 / 2] : (q,
    \tau) \in A_t \}) \m (\d q \d \tau)\\
    & \overset{\eqref{eq:expressAt}}{=} & \int_{S \times
    \mathbb{R}} \mathcal{L}^1 (\{ t\in [0, 1 / 2]:(q,\frac{\tau - at}{1 - t})\in A \})\m(\d q \d
    \tau) \overset{\eqref{eq:vanishingmassonray}}{=} 0.
  \end{eqnarray*}
  
  ({\romannumeral 3}) Prove that $h_q$ is positive and Lipschitz.
  For any $[a_0, b_0] \subset [a, b]$ and $S_0 \subset S$ with $\mathfrak{q}
  (S_0) > 0$, take $A$ as a bounded subset of $\mathcal{T}^b_u \cap (S_0 \times\R)$ having positive mass and consider the transport optimally moving $\m (A)^{- 1}\m\llcorner A$ into $S \times\{ (a_0 + b_0)/2\}$. 
  As in ({\romannumeral 2}), once $\m (A) > 0$, $\m (A_t) > 0$ for all $t \in
  (0, 1)$. One can easily find $A_t \subset S_0 \times [a_0, b_0]$ for some
  time $t$ from the boundedness of $A$, ensuring $\m (S_0 \times [a_0, b_0]) > 0$.
  
  Next, taking any such $A = S_0 \times [a_0, b_0]$, apply \eqref{eq:continuous mass of evolution set} to the optimal dynamical
  plan $\nu$ transporting $\mu_0 \coloneqq \m (A)^{- 1}\m\llcorner A$ into $S \times \{ b \}$. 
  By disintegration, we have
  \begin{equation} 
  \m (A_t) = \int_{S_0} \int_{[a_0, b_0]_t} h_q
     \d\mathcal{L}^1  \d\mathfrak{q}, \quad [a_0, b_0]_t \coloneqq [(1
     - t) a_0 + t b, (1 - t) b_0 + t b] . \end{equation}
  Given any $0 \leq r < s < 1$, the arbitrariness of $S_0$ implies, for
  $\mathfrak{q}$-a.e. $q \in S$
  \begin{equation}
    c (r, s) \int_{[a_0, b_0]_r} h_q \d\mathcal{L}^1 \leq \int_{[a_0, b_0]_s}
    h_q \d\mathcal{L}^1 \leq C (r, s)  \int_{[a_0, b_0]_r} h_q \d\mathcal{L}^1,
    \label{eq:ineqhq}
  \end{equation}
  where $c (r, s), C (r, s)$ are locally Lipschitz functions of $r, s$ given
  by \eqref{eq:continuous mass of evolution set}. Since $h_q$
  (or $\m_q$) is locally finite, both sides of \eqref{eq:ineqhq} continuously depend on $s, r, a_0, b_0$, and so \eqref{eq:ineqhq} holds simultaneously for all $r, s, a_0,
  b_0$. At Lebesgue points $\tau_0, \tau_1$ of $h_q$, choosing $[a_0, b_0]
  = [\tau_0 - \varepsilon, \tau_0 + \varepsilon]$, $r = 0$ and $s =
  \frac{\tau_1 - \tau_0}{b - \tau_0}$ in \eqref{eq:ineqhq} and
  shrinking $\varepsilon \rightarrow 0$, a two-sided inequality between $h_q
  (\tau_0)$, $h_q (\tau_1)$ follows, leading to the Lipschitz continuity.
  
  Finally, because $c (r, s), C (r, s)$ are positive for all $s, r$, the
  continuous density $h_q$ is either identically 0 or everywhere positive
  inside $I_q$. But the positivity of all $\m (S_0 \times [a, b])$
  excludes the former case (up to a $\mathfrak{q}$-negligible set of $q$).
  
  ({\romannumeral 4}). Prove that $h_q$ is a $\cd (K, N)$ density.
  Consider, any $a < A_0 < A_1 < b$ and $L_0, L_1 > 0$ with $A_0 + L_0 < A_1$
  and $A_1 + L_1 < b$. Define
  \begin{equation} \mu_0 \coloneqq \int_S \frac{1}{L_0} \mathcal{L}^1 \llcorner [A_0, A_0 +
     L_0] \mathfrak{q} (\d q), \qquad \mu_1 \coloneqq \int_S \frac{1}{L_1}
     \mathcal{L}^1 \llcorner [A_1, A_1 + L_1] \mathfrak{q} (\d q) . \end{equation}
  In ({\romannumeral 3}), we have shown the positivity of $h_q$, so densities of $\mu_i$
  w.r.t. $\m$ are
  \begin{equation} 
  \rho_i ((q, t))=\frac{1}{L_i} h_q (t)^{- 1},\quad \forall t \in [A_i,
     A_i + L_i],i = 0, 1. 
     \end{equation}
  When $L_0$ and $A_1 + L_1$ are close enough (up to further localizing $S$),
  we can apply $\cd_{\loc} (K, N)$ to the optimal dynamical plan
  between $\mu_0$ and $\mu_1$. 
  Then \eqref{eq:CDKNdensity} follows by the same argument in {\cite[Theorem 4.2]{CavallettiMondinoInven}}.
\end{proof}

\noindent \textbf{Acknowledgement.}
The author would like to thank Karl-Theodor Sturm for supervision during this project as well as Matthias Erbar and Timo Schultz for valuable comments.
\iffalse
\noindent \textbf{Declarations}

\noindent \textbf{Data availability} Data sharing not applicable to this article as no datasets were generated or analysed during the current study.

\noindent\textbf{Conflict of interest} On behalf of all authors, the corresponding author states that there is no conflict of interest.
\fi

\bibliographystyle{amsplain}
\bibliography{bib.bib}
\end{document}